\documentclass[10pt,oneside,leqno]{amsart}
\usepackage{amsfonts}
\usepackage{graphics,graphicx}
\usepackage{caption} 
\usepackage{epsfig}
\usepackage[utf8]{inputenc}
\usepackage[T1]{fontenc}
\usepackage{color}
\usepackage{amssymb}
\usepackage{mathrsfs}
\usepackage{amsthm}
\usepackage{amsmath}
\usepackage{amssymb}
\usepackage{latexsym}
\usepackage{cancel}
\usepackage{dsfont}
\makeatletter
\@addtoreset{equation}{section} 
\makeatother
\newtheorem{lem}{Lemma}[section]

\newtheorem{cor}{Corollary}[section]
\newtheorem{thm}{Theorem}[section]
\newtheorem{rem}{Remark}[section]
\newfont{\sBlackboard}{msbm10 scaled 900}

\newcommand{\dd}     {{\rm d}}
\newcommand{\mylabel}[1]{\label{#1}
            \ifx\undefined\stillediting
            \else \fbox{$#1$}\fi }
\newcommand{\BE}{\begin{equation}}

\newcommand{\EEQ}{\end{equation}}
\newcommand{\rfb}[1]{\mbox{\rm
   (\ref{#1})}\ifx\undefined\stillediting\else:\fbox{$#1$}\fi}

\newfont{\Blackboard}{msbm10 scaled 1200}
\newcommand{\bl}[1]{\mbox{\Blackboard #1}}
\newfont{\roma}{cmr10 scaled 1200}

\def\CC{\rm \hbox{C\kern-.56em\raise.4ex
         \hbox{$\scriptscriptstyle |$}\kern+0.5 em }}
\newcommand{\ud}{\mathrm{d}}
\newcommand{\be}{\begin{equation}}
\newcommand{\ee}{\end{equation}}
\newcommand{\beq}{\begin{eqnarray}}
\newcommand{\eeq}{\end{eqnarray}}
\newcommand{\beqs}{\begin{eqnarray*}}
\newcommand{\eeqs}{\end{eqnarray*}}
\newcommand{\bt}{\begin{thm}}
\newcommand{\et}{\end{thm}}
\newcommand{\br}{\begin{remark}}
\newcommand{\er}{\end{remark}}
\newcommand{\bc}{\begin{cor}}
\newcommand{\ec}{\end{cor}}
\newcommand{\el}{\end{lem}}
\newcommand{\bd}{\begin{definition}}
\newcommand{\ed}{\end{definition}}

\def\b{\beta}

\def\o{\omega}
\def\G{\Gamma}

\def\O{\Omega}

\def\cH{{\cal H}}

\newcommand{\mm}    {{\hbox{\hskip 0.5pt}}}

\newcommand{\bluff} {{\hbox{\raise 15pt \hbox{\mm}}}}
\newcommand{\bb}   {{\hbox{\fourteeni b}}}

\makeatletter
\def\section{\@startsection {section}{1}{\z@}{-3.5ex plus -1ex minus
    -.2ex}{2.3ex plus .2ex}{\large\bf}}
\makeatother
\def\be{\begin{equation}}
\def\ee{\end{equation}}

\def\ds{\displaystyle}

\newcommand{\R}{\bl{R}}

\begin{document}

\thispagestyle{empty}
\title[Coupled plate equations with singular structural damping]{Regularity and stability of two coupled Euler-Bernoulli equations with a localized singular structural damping}
\author{Ka\"{\i}s AMMARI}
\address{LR Analysis and Control of Pde, LR 22ES03, Department of Mathematics, Faculty of Sciences of Monastir, University of Monastir, 5019 Monastir, Tunisia} 
\email{kais.ammari@fsm.rnu.tn} 

\author{Fathi HASSINE}
\address{LR Analysis and Control of PDEs, LR 22ES03, Department of Mathematics, Faculty of Sciences of Monastir, University of Monastir, 5019 Monastir, Tunisia
\and
Higher Institute of Applied Mathematics and Computer Science, University of Kairouan, 3100 Kairouan, Tunisia}
\email{fathi.hassine@fsm.rnu.tn}

\author{Louis Tebou}
\address{Department of  Mathematics  and Statistics, College of Arts and Sciences, Florida International University, Modesto Maidique Campus, 
Miami, FL 33199, USA}
\email{teboul@fiu.edu}


\begin{abstract}
This paper is concerned with the study of regularity and stability properties of two Euler–Bernoulli beam equations with localized singular damping. Under suitable regularity assumptions on the damping coefficient, we establish Gevrey regularity for the semigroup generated by the associated operator. Furthermore, for a broader class of damping mechanisms, including less regular damping, we derive uniform 
stability result. These findings provide a detailed description of the long-term behavior of the corresponding dynamical systems.
\end{abstract}


\maketitle

\tableofcontents

\section{Introduction and statements of main results}
\setcounter{equation}{0}
In 1982, G. Chen and D.L. Russell \cite{cru} considered the following model of damped elastic system
 $$\ddot x+Ax+B\dot x=0,\quad t>0$$ on a Hilbert space $X$, where $A$ is a positive self-adjoint operator on its domain $D(A)$ and $D(A)$ is dense in $X$. The damping operator $B$ is also positive self-adjoint  and its domain $D(B)$ is dense in $X$. \\  
They proved that the underlying semigroup is analytic, therefore exponentially stable, when $B$ is proportional to either $A^{1/2}$ or $A$. Then, 
they  raised the following two conjectures: The semigroup is analytic if 

\begin{enumerate}
\item there exist constants $\rho_1>0$ and $\rho_2>0$: $$\rho_1 A^{\frac{1}{2}}\leq B\leq \rho_2 A^{\frac{1}{2}}$$
\item {  there exist constants $\rho_1>0$ and $\rho_2>0$: $$\rho_1^2 A\leq B^2\leq \rho_2^2 A.$$}
\end{enumerate}
In the same paper, they also proposed the following partial answers:
\begin{itemize}
\item   they proved the following result: The semigroup analytic if 
$$\forall \rho>0,\,\exists\varepsilon(\rho)>0:[2\rho-\varepsilon(\rho)]A^{\frac{1}{2}}\leq B\leq [2\rho+\varepsilon(\rho)]A^{\frac{1}{2}},$$
\item they also provided a partial answer to the 2nd conjecture provided extra conditions that are difficult to check were imposed.
\end{itemize}
Complete solutions to those conjectures were later provided by S. Chen and Triggiani:
\begin{itemize}
\item { In 1989, S. Chen and  Triggiani \cite{cht2},  under the conditions:
$$\text{ There exist }  \frac{1}{2}\leq \alpha\leq1\text{ and } 0<\rho_1<\rho_2<\infty:\rho_1 A^\alpha\leq B\leq \rho_2 A^\alpha$$ proved that the semigroup is analytic.}
 \\ In the same paper, they also proved that the semigroup fails to be analytic when $0<\alpha<1/2$, but is differentiable.
\item  In 1990,  they further investigated the case $0<\alpha<1/2$ in \cite{cht3}, and proved that the semigroup is of Gevrey class $s$ for every $s>\frac{1}{2\alpha};$
so the semigroup is infinitely differentiable for every $t>0$ and exponentially stable.\end{itemize}
It is worth mentioning that Huang \cite{hfa} as well as Taylor \cite{taylor} also contributed to the solutions of those conjectures.\\  Subsequently, in the particular case of a model of plate equation sandwiched between the Euler-Bernoulli plate model and the Kirchhoff plate model, the semigroup regularity and stability was discussed in \cite{terpl}, thereby extending the findings of those earlier works to this case. It is also worth mentioning the works, e.g. \cite{lre,lta,ltd,ktw,kls} where similar extensions were established in the framework of thermoelastic plates. \\ The problem to be tackled in this paper falls within the class of simultaneous control. Simultaneous control of a system of two or more equations involves controlling the system using the same control device for all the components of this system. This notion in the framework of the controllability of distributed systems was introduced by Russell in his study of the boundary controllability of Maxwell equations in rectangular domains \cite{rum}. To solve that problem, he transformed it into a controllability problem for a system of two uncoupled wave equations, one having the Dirichlet boundary conditions while the other one having the Neumann boundary conditions. That simultaneous controllability result was later generalized by Lions in the first volume of his monograph on controllability to a larger class of domains, and to uncoupled plate equations, \cite{lioc}. As for the simultaneous internal controllability of uncoupled wave equations (two or more equations) with different speeds of propagation, Haraux initiated that work and he established some unique continuation results leading to approximate controllability in all space dimensions, \cite{hao}. He also proved an exact controllability result in one space dimension where the control region was an arbitrary open subinterval of the interval under consideration, and he proved another one in higher space dimensions where the control region was the whole domain under consideration. The one-dimensional exact controllability result of haraux was  generalized to all space dimensions in \cite{tebb}, and under the Bardos-Lebeau-Rauch geometric control condition ``The open set $\omega$ is an admissible control support for exact controllability in time $ T$ if every ray of geometric optics enters $\omega$ in a time less than $T$”. \\  Recently  the following  abstract system with a simultaneous damping mechanism was considered in \cite{AST1}:
\begin{equation}\label{e1}\begin{array}{lll}
&y_{tt}+aAy+\gamma A^{\theta}(y_t+z_t)=0\text{ in }(0,\infty)\\
&z_{tt}+bAz+\gamma A^{\theta}(y_t+z_t)=0\text{ in }(0,\infty)\\
&y(0)=y^0\in V,\quad y_t(0)=y^1\in H,\quad z(0)=z^0\in V,\quad z_t(0) =z^1 \in H,\end{array}\end{equation} 
where $a,~b$, $\gamma$ are positive constants with $a\not= b$, and $\theta\in[-1,1]$; $H$ is a Hilbert space with inner product $(.,.)$ and norm $|.|$,  $A$  is a positive unbounded self-adjoint operator, with domain $D(A)$ dense in the Hilbert space $H$;
 $V=D(A^\frac{1}{2})$ with  $V\hookrightarrow H\hookrightarrow V'$, each injection being dense and compact, where $V'$ denotes the topological dual of $V$.\\  The authors of \cite{AST1} showed that the corresponding semigroup is 
 \begin{itemize}
 \item differentiable for all $0<\theta<1$, but is not analytic for $1/2<\theta\leq1$; in particular, the semigroup is exponentially stable for all $0\leq\theta\leq1$;
 \item of  Gevrey class $s$ for every  $s>1/2\theta$, for $0<\theta\leq1/4$, and of Gevrey class $s$ for every $s>(1+2\theta)/3\theta$, for $1/4<\theta\leq1/2$;
 \item polynomially stable for $-1\leq\theta<0$; more precisely the semigroup is $O(t^{\frac{1}{2\theta}})$ as $t$ goes to infinity, and this decay rate is optimal.
 \end{itemize}
 Later on and for the same system \eqref{e1}, in \cite{klt}, the authors improved the regularity of the semigroup as follows:
 \begin{itemize}
 \item the semigroup is analytic for $\theta=1/2$;
 \item the semigroup is  of  Gevrey class $s$ for every  $s>1/2\theta$, for $0<\theta<1/2$, and of Gevrey class $s$ for every $s>1/2(1-\theta)$, for $1/2<\theta<1.$ In particular, the analysis of the spectrum and \cite[Theorem 2, p.147 ]{taylor} show that those regularity results are optimal.
 \end{itemize}
 The type of system that we shall consider in this work is similar to \eqref{e1} in the sense that the system is also velocity coupled; however, now, the damping mechanism is no longer global, but rather localized in a proper subset of the domain under consideration. This makes the investigation of stability and regularity more challenging than in the case of global damping mechanisms. This justifies our interest in this problem.\\
 In the sequel, we shall use the following notation: Let $\Omega \subset \R^{n}$, $n \geq 2,$ be a bounded domain with a sufficiently smooth boundary $\Gamma=\partial \Omega$. 

Consider the damped plate system
\begin{equation}
\label{wave1}
\partial_t^2 u + d \, \Delta^2 u - \, \mathrm{div}(a(x) \, \nabla (\partial_t u + \partial_t v)) = 0\text{ in } \Omega \times (0,+\infty), 
\end{equation}
\begin{equation}
\label{wave1bis}
\partial_t^2 v +c \, \Delta^2 v - \, \mathrm{div}(a(x) \, \nabla (\partial_t u + \partial_t v)) = 0\text{ in } \Omega \times (0,+\infty), 
\end{equation}
\begin{equation}
\label{wave2}
u =\partial_\nu u= v =\partial_\nu v= 0\text{ on } \partial \Omega \times (0,+\infty),
\end{equation}
\begin{equation}
\label{wave3}
u(x,0) = u^0(x), \, v(x,0) = v^0(x), \partial_t u(x,0) = u^1 (x), \partial_t v(x,0) = v^1 (x)\text{ in } \Omega,
\end{equation}
where $c \neq d>0$ are constants and $a \in L^\infty (\Omega),$ is nonnegative in $\Omega$ and positive in $\omega$, where $\omega$ is an arbitrary nonempty open subset of $\Omega$.\\
The natural energy of the solution of \rfb{wave1}-\rfb{wave3} at time $t$ is given by
\begin{equation*}
E(u,v,t)=
\frac{1}{2} \int_\Omega\left( \left|\partial_t u(x,t)\right|^2 + \left|\partial_t v(x,t)\right|^2 + d \, \left|\Delta u(x,t)\right|^2 + 
c \, \left|\Delta v(x,t)\right|^2 \right) \, \ud x, \, \forall \, t \geq 0.
\end{equation*}
Simple formal calculations yield
\begin{equation*}
E(u,v,t)-E(u,v,s)= - \, \int_{s}^{t} \int_{\Omega} a(x) \, \left|\nabla(\partial_{t} u(x,s) + \partial_t v(x,s))\right|^2 \,\ud x\,\ud s,\forall t>s\geq 0, 
\end{equation*}
and therefore, the energy is a nonincreasing function of the time variable $t$.\\ The questions that we would like to address in this note are the following:
\begin{itemize}
    \item Does the underlying semigroup decay to zero as the time variable $t$ goes to infinity? If so, what is its decay rate? 
    \item What is the regularity of the semigroup? Given the abstract result of \cite{klt}, we know that this semigroup is analytic when the support $\o$ of the feedback control is the whole domain $\Omega$. Thanks to a recent semigroup regularity result of \cite{lasteb}, we know that for a single Euler-Bernoulli plate with localized structural damping, the semigroup is of Gevrey class $s$ for every $s>5/2$ provided the proper subset $\omega$ is big enough, (a simple example is the case where $\omega$ is a collar around the whole boundary of $\Omega$. It then makes sense to reformulate the regularity question as: Is the semigroup for \eqref{wave1bis}-\eqref{wave3} of a certain Gevrey class? If so, which one? 
\end{itemize}
 Before attempting to answer those questions, we want to draw the reader's attention to the importance of different speeds of propagation. Indeed, if $c=d$, then the system is unstable; to see this, assume $c=d$ and set $p=u+v$ and $q=u-v$, then the functions $p$ and $q$ satisfy

\begin{equation}
\label{wave11}
\partial_t^2 p + c \, \Delta^2 p - \, 2\mathrm{div}(a(x) \, \nabla (\partial_t p)) = 0\text{ in } \Omega \times (0,+\infty), 
\end{equation}
\begin{equation}
\label{wave11bis}
\partial_t^2 q + c \, \Delta^2 q = 0\text{ in } \Omega \times (0,+\infty), 
\end{equation}
\begin{equation}
\label{wave2a}
p =\partial_\nu p= q =\partial_\nu q= 0\text{ on } \partial \Omega \times (0,+\infty),
\end{equation}
\begin{align}\label{wave3a}
&p(x,0) = u^0(x)+v^0(x), \, q(x,0) =u^0(x)- v^0(x)\text{ in } \Omega,\notag\\& \partial_t p(x,0) = u^1 (x)+v^1(x), \partial_t q(x,0) = u^1(x)-v^1 (x)\text{ in } \Omega.
\end{align}
Notice that the $(p,q)-$system is uncoupled and the energy of the $q-$ system is conserved, while one can show that the energy of the $p-$system decays to zero. Thus, the $(u,v)-$ system is unstable since it will vibrate continuously, \cite{ahkht}.\\
Before stating our results, we start by introducing 
 the energy space by $\mathcal{H}=(H_{0}^{2}(\Omega))^2\times (L^{2}(\Omega))^2$ which is endowed with the usual inner product
$$
\left\langle(u_{1},u_2,u_3,u_4);(v_{1},v_{2},v_3,v_4)\right\rangle=  
$$
$$
d \, \int_{\Omega}\Delta u_{1}(x)\Delta \overline{v}_{1}(x)\,\dd x+
c \, \int_{\Omega}\Delta u_{2}(x)\Delta \overline{v}_{2}(x)\,\dd x+ \int_{\Omega}u_{3}(x)\overline{v}_{3}(x)\,\dd x + \int_{\Omega}u_{4}(x)\overline{v}_{4}(x)\,\dd x.
$$
Next, we shall define the linear unbounded operator $\mathcal{A}:\mathcal{D}(\mathcal{A})\subset\mathcal{H}\longrightarrow\mathcal{H}$ by
\begin{align*}
\mathcal{D}(\mathcal{A})=\{(u_1,u_2,v_1,v_2)\in\mathcal{H}: v_1,v_2 \in H_{0}^{2}(\Omega),\; d \Delta^2 u_1-\mathrm{div}(a\nabla (v_1+ v_2))\in L^{2}(\Omega), 
\\
c \Delta^2 u_2 -\mathrm{div}(a\nabla (v_1+ v_2)) \in L^{2}(\Omega)\}
\end{align*}
and
$$
\mathcal{A}(u_1,u_2,v_1,v_2)^{t} = (v_1, v_2, -d \, \Delta^2 u_1 + \mathrm{div}(a\nabla (v_1+ v_2)), 
-c \, \Delta^2 u_2 + \mathrm{div}(a \nabla (v_1+ v_2)))^{t}.
$$
 We also recall the following geometric condition that will be useful in the proofs of exponential stability and regularity results.\\
Introduce a geometric constraint (GC) on the subset
$\o$ where the dissipation is effective; we proceed as in 
	\cite{lu}, (see also \cite{kb, lioc} for a special case).\\

\let\cal=\mathcal\let\p=\partial \let\d=\delta
\noindent
{\bf (GC).} There exist open sets $\Omega_j\subset\O$ with piecewise
smooth boundary $\partial\Omega_j$, and points $x^j_0\in {\mathbb R}^{n}$,
$j=1,2,\dots, J$, such that $\Omega_i\cap\Omega_j=\emptyset$, for any $1\leq
i<j\leq J$, and:
$$\Omega\cap{\cal
N}_\d\left[\left(\bigcup_{j=1}^J\G_j\right)\bigcup\left(\Omega\setminus\bigcup_{j=1}
^J\Omega_j\right)\right]\subset\o,$$ for some $\d>0$, where
${\cal N}_\d(S)=\displaystyle\bigcup_{x\in S}\{y\in{\mathbb R}^n;|x-y|<\d\},\hbox{
for }S\subset{\mathbb
R}^n,$\\ $\Gamma_j=\left\{x\in\p\Omega_j;(x-x_0^j)\cdot\nu^j(x)>0\right\},$ $\nu^j$
being the unit normal vector pointing into the exterior of $\Omega_j$. 

\medskip

Let $0<\d_1<\d$. Set $Q_1={\cal N}_{\d_1}(S)$ and $\omega_1=\Omega\cap Q_1$.
 
Our main result is a regularity result for the corresponding semigroup, and it reads:
\begin{thm}\label{reg}
 Suppose now that $\o$ satisfies the geometric constraint {\bf (GC)}.  Further, assume that the function $a$ lies in $C^2(\bar\Omega)$, vanishes in $\Omega\setminus\omega$, and satisfies:
\begin{equation}\label{sa}\begin{array}{ll}
&\exists\, M>0:|\nabla a(x)|^4\leq M(a(x))^3 \text{ and } |D^2 a(x)|^2\leq Ma(x),\quad\forall\,x\in\Omega,\\~\cr& \text{there exists an open set }\omega_1\Subset\omega,~\exists\, a_0>0:a(x)\geq a_0,\quad\forall\, x\in\omega_1.\end{array}\end{equation} where $D^2a(x)$ denotes the Hessian matrix of $a$ at $x$.\\ Then, the semigroup $(S(t))_{t\geq0}$ is of Gevrey class $s$ for every $s>5/2$, as its resolvent satisfies the following estimate
\begin{align}\label{gevre}
\exists\, C_0>0:\,|\lambda|^{\frac{2}{5}}\left\Vert(i\lambda I-{\cal A})^{-1}\right\Vert_{{\cal L}({\cal H})}\leq C_0,\quad\forall\,\lambda\in{\mathbb R}\text{ with }|\lambda| \text{ large enough}.
\end{align}
\\
Therefore, the semigroup $(S(t))_{t\geq0}$ is infinitely differentiable at $t$ for all $t>0$.
\end{thm}

\begin{rem} 
This regularity result shows that the Gevrey class established in \cite{lasteb} for a single Euler-Bernoulli equation is the same for the singular system that we are dealing with. The proof of Theorem \ref{reg} is delicate; given that the matrix defining the simultaneous damping mechanism is singular, to guarantee the claimed regularity, one should be careful when estimating the localized energy, and also when propagating the damping effect from the damping region $\omega$ to the whole domain $\Omega$. \\ An example of a function $a$ satisfying the constraints \eqref{sa} was constructed in \cite[Remark 1.5]{lasteb}.  
\end{rem}
 
Now, we are going to state an exponential stability result for the semigroup. Notice that the resolvent estimate in Theorem \ref{reg} combined with the strong stability of the semigroup leads to the exponential stability of the semigroup, provided the damping coefficient is smooth enough and satisfies the constraints \eqref{sa}. The result we shall now state does not impose any regularity on the damping coefficient:

\begin{thm}\label{expstab}
 Suppose now that $\o$ satisfies the geometric constraint {\bf (GC)}.  Further, assume that the nonnegative function $a$ lies in $L^\infty(\Omega)$, and satisfies:
\begin{equation}\label{coerc}
\exists\, a_0>0:a(x)\geq a_0,\quad\text{ for a.e. } x\in\omega.\end{equation}  Then, the energy of System \eqref{wave1}-\eqref{wave3} decays exponentially, {\tt viz.}, there exist positive constants $C_0$ and $\mu$ such that for every $(u^0,u^1,v^0,v^1)$ in ${\cal H}$, one has the following inequality:
\begin{align}\label{enedec}
E(t)\leq C_0 e^{-\mu t}E(0),\quad\forall\, t\geq0.
\end{align}
\end{thm}

\begin{rem}
The exponential decay of the energy in Theorem \ref{expstab} should be contrasted with the polynomial decay of the energy established in the case two wave equations similarly damped, even with a smooth damping coefficient, \cite{ahkht}. Indeed, in the case of the wave equations, the localized simultaneous damping mechanism is Kelvin-Voigt type, and the corresponding semigroup is not differentiable even when this damping mechanism is globally distributed in $\Omega$, \cite{klt}. This lack of regularity is at the root of the weaker decay of the energy in the case of wave equations. In the case at hand, the damping mechanism is structural or square root type, and it is known that the semigroup is analytic when the damping is globally distributed. It would be interesting to examine what happens for a system of two Euler-Bernoulli plates when the localized simultaneous structural damping is replaced by a localized simultaneous Kelvin-Voigt damping; in this regard, we refer the interested reader to the recent work \cite{lasteb2}, where the authors proved some regularity as well as stability results for the semigroup corresponding to a single Euler-Bernoulli plate with localized Kelvin-Voigt damping.
\end{rem}

The remainder of this paper is organized as follows. In Section \ref{wellposed}, we prove the well-posedness of the system \rfb{wave1}-\rfb{wave3} as well as its strong stability. In section \ref{greg}, we show the claimed Gevrey regularity. Section \ref{pexpstab} is devoted to the proof of Theorem \ref{expstab}. 

\section{Well-posedness and strong stability}\label{wellposed}
We can write \eqref{wave1}-\eqref{wave3} as the following Cauchy problem
$$
\left\{\begin{array}{l}
\ds\frac{d}{dt}(u(t),v(t),\partial_t u(t), \partial_t v (t))^{t}=\mathcal{A}(u(t),v(t),\partial_t u(t), \partial_t v (t))^{t},
\\
(u(0),v(0),\partial_t u(0), \partial_t v (0))=(u^{0},u^{1},v^0,v^1).
\end{array}\right.
$$
\begin{thm}
The operator $\mathcal{A}$ generates a $C_{0}$-semigroup of contractions on the energy space $\mathcal{H}$.
\end{thm}
\begin{proof}
Firstly, it is easy to see that for all $(u_1,u_2,v_1,v_2)\in\mathcal{D}(\mathcal{A})$, we have
$$
\mathrm{Re}\left\langle\mathcal{A}(u_1,u_2,v_1,v_2);(u_1,u_2,v_1,v_2)\right\rangle=-\int_{\Omega}a(x) \, |\nabla (v_1(x) + v_2(x))|^{2}\,\dd x,
$$
which show that the operator $\mathcal{A}$ is dissipative.

Next, for any given $(f_1,f_2,g_1,g_2)\in\mathcal{H}$, we solve the equation $\mathcal{A}(u_1,u_2,v_1,v_2)=(f_1,f_2,g_1,g_2)$, which is recast on the following way
\begin{equation}\label{WPwave}
\left\{\begin{array}{l}
v_1=f_1,
\\
v_2 = f_2,
\\
- d \Delta^2 u_1+\mathrm{div}(a\nabla (f_1+f_2))=g_1,
\\
- c \Delta^2 u_2+\mathrm{div}(a\nabla (f_1 + f_2))=g_2.
\end{array}\right.
\end{equation}
It is well known that by Lax-Milgram's theorem the system \eqref{WPwave} admits a unique solution $u\in H_{0}^{2}(\Omega)$. Moreover by multiplying the second line of \eqref{WPwave} by $\overline{u}$ and integrating over $\Omega$ and using Poincar\'e inequality and Cauchy-Schwarz inequality we find that there exists a constant $C>0$ such that
$$
d\, \int_{\Omega}|\Delta u_1(x)|^{2}\,\dd x +c \, \int_{\Omega}|\Delta u_2(x)|^{2}\,\dd x $$
$$
\leq C\left(\int_{\Omega}|\nabla f_1(x)|^{2}\,\dd x+\int_{\Omega}|\nabla f_2(x)|^{2}\,\dd x +\int_{\Omega}|\nabla g_1(x)|^{2}\,\dd x + \int_{\Omega}|g_2(x)|^{2}\,\dd x\right).
$$
It follows that for all $(u,v)\in\mathcal{D}(\mathcal{A})$ we have
$$
\|(u_1,u_2,v_1,v_2)\|_{\mathcal{H}}\leq C\|(f_1,f_2,g_1,g_2)\|_{\mathcal{H}}.
$$
This implies that $0\in\rho(\mathcal{A})$ and by contraction principle, we easily get $R(\lambda\mathrm{I}-\mathcal{A})=\mathcal{H}$ for sufficient small $\lambda>0$. The density of the domain of $\mathcal{A}$ follows from \cite[Theorem 1.4.6]{Pazy}. Then thanks to Lumer-Phillips Theorem (see \cite[Theorem 1.4.3]{Pazy}), the operator $\mathcal{A}$ generates a $C_{0}$-semigroup of contractions on the Hilbert $\mathcal{H}$. 
\end{proof}
\begin{thm}
The semigroup $e^{t\mathcal{A}}$ is strongly stable in the energy space $\mathcal{H}$, i.e,
$$
\lim_{t\to+\infty}\|e^{t\mathcal{A}}(u^{0},v^{0},u^1,v^1)^{t}\|_{\mathcal{H}}=0,\;\forall\,(u^{0},v^{0},u^1,v^1)\in\mathcal{H}.
$$ 
\end{thm}
\begin{proof}
Since the resolvent of the operator $\mathcal{A}$ is compact, it follows from \cite{AB,FL} that, in order to establish the strong stability of the semigroup $(e^{t\mathcal{A}})_{t \geq 0}$, it suffices to show that
$
\sigma(\mathcal{A}) \cap i\mathbb{R} = \varnothing.
$
Moreover, since $0 \in \rho(\mathcal{A})$, it remains to prove that the operator $(i\mu I - \mathcal{A})$ is injective for all $\mu \in \mathbb{R}^*$.

\medskip

Let $(u_1,u_2,v_1,v_2)\in\mathcal{D}(\mathcal{A})$ such that 
\begin{equation}\label{Iwave}
\mathcal{A}(u_1,u_2,v_1,v_2)^{t}=i\mu(u_1,u_2,v_1,v_2)^{t}.
\end{equation}
Then taking the real part of the scalar product of \eqref{Iwave} with $(u,v)$ we get
$$
\mathrm{Re}(i\mu\|(u_1,u_2,v_1,v_2)\|_{\mathcal{H}}^{2})=\mathrm{Re}\left\langle\mathcal{A}(u_1,u_2,v_1,v_2),(u_1,u_2,v_1,v_2)\right\rangle=-\int_{\Omega} a |\nabla (v_1+v_2)|^{2}\dd x=0.
$$
which implies that
\begin{equation}\label{Dwave}
a \, \nabla (v_1 + v_2)=0 \qquad \text{ in }\,\Omega.
\end{equation}
Inserting \eqref{Dwave} into \eqref{Iwave}, we obtain
\begin{equation}\label{waveI1}
\left\{\begin{array}{ll}
\mu^{2}u_1 - d \Delta^2 u_1=0&\text{in }\Omega, 
\\
\mu^{2}u_2 - c \Delta^2 u_2=0&\text{in }\Omega,
\\
\nabla (u_1 + u_2) =0&\text{in }\omega
\\
u_1 = u_2 =0, \partial_\nu u_1 = \partial_\nu u_2 = 0&\text{on }\Gamma.
\end{array}\right.
\end{equation}

We denote by $w^i_{j}=\partial_{x_{j}}u_{i}, i=1,2$ and we derive the equations of \eqref{waveI1}, one gets
\begin{equation*}
\left\{\begin{array}{ll}
\mu^{2}w^1_{j} - d \Delta^2 w^1_{j}=0&\text{in }\Omega,\\
\mu^{2}w^2_{j} - c \Delta^2 w^2_{j}=0&\text{in }\Omega,
\\
w^1_{j} + w^2_j =0&\text{in }\omega.
\end{array}\right.
\end{equation*}
Which implies that
$$
d \Delta^2 w^1_{j} + c \Delta^2 w^2_{j} =0 \, \text{in }\omega  \Longrightarrow  (d - c) \Delta^2 w^i_{j} =0, \, \text{in }\omega, \, i=1,2 \Longrightarrow \Delta^2 w^i_{j} =0 \, \text{in }\omega, i=1,2.   
$$
According to the above system we have that  $w^i_{j} =0 \, \text{in }\omega, i=1,2$ and

\begin{equation*}
\left\{\begin{array}{ll}
\mu^{2}w^1_{j} - d \Delta^2 w^1_{j}=0&\text{in }\Omega,\\
\mu^{2}w^2_{j} - c \Delta^2 w^2_{j}=0&\text{in }\Omega,
\\
w^1_{j} = w^2_j =0&\text{in }\omega.
\end{array}\right.
\end{equation*}

Hence, from the unique continuation theorem we deduce that $w^i_{j}=0$ in $\Omega, i=1,2$ and therefore $u_1,u_2$ are constants in $\Omega$ and since ${u_i}_{|\Gamma} = {\partial_\nu u_i}_{|\Gamma} = 0, i=1,2$ we follow that $u_i\equiv 0, i=1,2$. We have thus proved that $\mathrm{Ker}(i\mu I-\mathcal{A})= \left\{0\right\}$. The proof is thus complete.
\end{proof}

 \let\bb=\mathbb \let\nag=\notag\let\nt=\noindent
\section{Proof of Theorem \ref{reg}}\label{greg}

In this section, we shall prove Theorem \ref{reg}. The main idea of the proof is to treat each equation separately and then suitably combine the estimates to derive the claimed Gevrey regularity.
For this proof, we will rely on Taylor's sufficient condition for Gevrey regularity, \cite[Theorem 4, p.155]{taylor}.
This proof will be completed as soon as we establish the following resolvent 
estimate \begin{equation}\label{res1}\|(i\lambda{\cal I}-{\cal A})^{-1}\|_{{\cal L}({\cal
H})}=O(|\lambda|^{-\frac{2}{5}}) \text{ as } |\lambda|\nearrow+\infty.\end{equation}
To this end, let $U\in {\cal H}$, and let $\lambda$ be a real number with $|\lambda|>1$.
Since the range of $i\lambda{\cal I}-{\cal A}$ is ${\cal H}$, there
exists $Z\in{ D}({\cal A})$ such that
\begin{equation}\label{xae0}i\lambda Z-{\cal A} Z=U.\end{equation} The desired resolvent estimate will be established once we prove
\begin{equation}\label{xae1}\|Z\|_{\cal H}\leq K_0|\lambda|^{-\frac{2}{5}}\,\|U\|_{\cal H},\end{equation}where here and in the sequel, $K_0$ is a generic positive constant that may eventually
depend on
$\O$, $\o$, and $a$, but never on $\lambda$.\\ \noindent To establish \eqref{xae1}, first, we note that
if $Z=(u,w,v,z)$ and $U=(f,h,g,\ell)$, then \eqref{xae0}\ may be recast as
\begin{align}\label{xae2}
&i\lambda u-v=f\nag\\&i\lambda v+d\Delta^2 u-\text{div}(a\nabla (v+z))=g\\
&i\lambda w-z=h\nag\\&i\lambda z+c\Delta^2 w-\text{div}(a\nabla (v+z))=\ell.\nag\end{align} First taking the inner product with $Z$
on both sides of \eqref{xae0}, then taking the real parts, we immediately
get
\begin{equation}\label{dislaw}|\sqrt{a}(\nabla v+\nabla z)|_2^2\leq \|U\|_{\cal H}\|Z\|_{\cal H}.\end{equation}
It follows from \eqref{dislaw}, the
first and third equations in \eqref{xae2}, as well as Poincar\'e inequality:
\begin{align}\label{dislaw1}\lambda^2|\sqrt{a}(\nabla u+\nabla w)|_2^2&\leq2|\sqrt{a}(\nabla v+\nabla z)|_2^2+2|\sqrt{a}(\nabla f+\nabla h)|_2^2\nag\\&\leq
2\|U\|_{\cal H}\|Z\|_{\cal H}+K_0\left(\|f\|_2^2+\|h\|_2^2\right)\leq
2\|U\|_{\cal H}\|Z\|_{\cal H}+K_0\|U\|_{\cal H}^2.\end{align}Henceforth, $|.|_2$ stands for $\|.\|_{L^2(\Omega)}$ or $\|.\|_{[L^2(\Omega)]^n}$, while $\|.\|_{s}$ stands for $\|.\|_{H^s(\Omega)}$ for every nonzero $s$ in ${\bb R}$.\\

\noindent
{\bf Methodology of the proof.} Due to the fact that the damping is localized, we will need first order or flow multipliers to propagate the effect of the damping mechanism from the damped region $\omega$ to the whole domain $\Omega$. We note that this damping propagation process is quite challenging for the following reasons: 
\begin{itemize}
\item the damping is unbounded and localized,
\item the matrix defining the damping is singular, making it quite tricky to combine the estimates established for $(u,v)$ and $(w,z)$-systems; this is why we need the flexural rigidity constants of the two plates to be distinct; 
\item our desire is to get the same Gevrey class for this singular system as for a single Euler-Bernoulli plate with localized structural damping established in \cite{lasteb}.
\end{itemize}
The rest of the proof will go along the following steps:\\
Step 1: Preliminary estimates. Here, we will focus on velocities with the ultimate goal of estimating the localized kinetic energy. The main idea is to express the localized kinetic energy using the sum of velocities, so that we can then invoke \eqref{dislaw}; \\
Step 2: Estimating the localized energy. Once the localized kinetic energy is estimated, we just rely on the local equipartition of energy identity and on grouping the terms involving the damping  to estimate the localized potential energy in a suitable way, (the grouping enables the canceling of undesirable terms) then derive an estimate for the localized energy;\\
Step 3: Geometric propagation of "regularity". In this step, we shall use appropriate flow multipliers to propagate the localized "regularity" from $\omega$ to the whole domain $\Omega$;\\
Step 4: Final estimate and completion of the proof.\\
\let\var=\varphi  \let\a=\alpha  \let\O=\Omega   \let\o=\omega  \let\nag=\notag \let\eqre=\eqref

\noindent
{\bf STEP 1: Estimating the localized kinetic energy.} First, we shall derive several estimates for the velocities $v$ and $z$.  Then, we shall use those estimates to derive an estimate of the term $|\sqrt{a}(v+z)|_2$. Afterward, we shall estimate the localized crossed product $\displaystyle\Re\int_\Omega av\bar z\,dx$. Finally, combining those two estimates we will get the sought after estimate of the localized kinetic energy of the coupled system.\\
The first step in this process is to decompose each velocity into two component: a very smooth one and a less smooth one; this follows an idea introduced in the literature by Liu and Renardy \cite{lre} when investigating the analyticity of the semigroup associated with an Euler-Bernoulli thermoelastic plate. This idea has been very useful in proving the regularity of the semigroup in the following works, e.g. \cite{AST1,ktw,lasteb,terpl}). Set
$$v=v_1+v_2,\text{ with }i\lambda v_1-\Delta v_1=g,\text{ and }i\lambda v_2=-d\Delta^2u+\text{div}(a\nabla (v+z))-\Delta v_1.$$ Similarly, set
$$z=v_1+z_2,\text{ with }i\lambda z_1-\Delta z_1=\ell,\text{ and }i\lambda z_2=-c\Delta^2w+\text{div}(a\nabla (v+z))-\Delta z_1.$$
On the one hand, the energy method readily shows
\begin{align}\label{vel1}& |\lambda||v_1|_2+|\lambda|^{\frac{1}{2}}|\nabla v_1|_2\leq K_0||U||_{\cal H}\nag\\&|\lambda||z_1|_2+|\lambda|^{\frac{1}{2}}|\nabla z_1|_2\leq K_0||U||_{\cal H}.\end{align}
On the other hand, we have
\begin{align}\label{vel2}|\lambda|||v_2||_{-2}\leq K_0|\Delta u|_2+||\hbox{div}(a\nabla( v+z))||_{-2}+K_0|v_1|_2. \end{align}
Given that  the function $a$ lies in $C^2(\bar\Omega)$, the application of Hahn-Banach theorem shows that there exists some function $\var$ in $H^2_0(\O)$ with $||\hbox{div}(a\nabla v)||_{-2}=||\var||_{2}$ such that
$$||\hbox{div}(a\nabla (v+z))||_{-2}^2=-\int_\O a\nabla (v+z)\cdot\nabla \bar\var\,dx=\int_\O (v+z)\nabla a\cdot\nabla\bar\var\,dx+\int_\O (v+z)a\Delta\bar\var\,dx.$$ It then readily follows
$$||\hbox{div}(a\nabla (v+z))||_{-2}\leq K_0(|v|_2+|z|_2).$$
Therefore, combining the last inequality with \eqref{vel2} and using \eqref{vel1}, we derive
\begin{align}\label{vel3}|\lambda|||v_2||_{-2}\leq K_0(||Z||_{\cal H}+|\lambda|^{-1}||U||_{\cal H}).\end{align}
In a similar way, it follows
\begin{align}\label{vel3b}|\lambda|||z_2||_{-2}\leq K_0(||Z||_{\cal H}+|\lambda|^{-1}||U||_{\cal H}).\end{align}
By interpolation, we also derive, thanks to \eqref{vel3} and \eqref{vel3b}:
\begin{align}\label{velxa3}&||v_2||_{-1}\leq K_0||v_2||_{-2}^{\frac{1}{2}}(|v|_2+|v_1|_2)^{\frac{1}{2}}\leq K_0|\lambda|^{-{\frac{1}{2}}}(||Z||_{\cal H}+|\lambda|^{-1}||U||_{\cal H})\nag\\&||z_2||_{-1}\leq K_0||z_2||_{-2}^{\frac{1}{2}}(|z|_2+|z_1|_2)^{\frac{1}{2}}\leq K_0|\lambda|^{-{\frac{1}{2}}}(||Z||_{\cal H}+|\lambda|^{-1}||U||_{\cal H}).\end{align}
Since the damping term involves the sum of the velocities, to estimate the localized kinetic energy, we start by estimating $|\sqrt{a}(v+z)|_2$. Using our decomposition, we find
\begin{align}\label{ez8}
|\sqrt{a}(v+z)|_2&\leq  |\sqrt{a}(v_1+z_1)|_2+|\sqrt{a}(v_2+z_2)|_2\nag\\&\leq K_0|\lambda|^{-1}||U||_{\cal H}+K_0||\sqrt{a}(v_2+z_2)||_{-2}^{1\over3}(||\sqrt{a}(v+z)||_1+||\sqrt{a}(v_1+z_1)||_1)^{\frac{2}{3}}.
\end{align}
Thanks to Poincar\'e inequality and the structural constraint on the function $a$, we have
\begin{align}\label{ez9}
||\sqrt{a}(v+z)||_1^2&=\int_\O\left|\frac{\nabla a}{2\sqrt{a}}(v+z)+\sqrt{a}\nabla (v+z)\right|^2\,dx\nag\\&\leq K_0\left(|v|_2^2+|z|_2^2+\int_\O a|\nabla (v+z)|^2\,dx\right)\nag\\&\leq
K_0\left(||Z||_{\cal H}^2+||Z||_{\cal H}||U||_{\cal H}\right),\text{ by invoking \eqref{dislaw}}.
\end{align}
Similarly, one readily checks
\begin{align}\label{ey1}
||\sqrt{a}v_1||_1^2+||\sqrt{a}z_1||_1^2\leq K_0(|v_1|_2^2+|z_1|_2^2)+K_0\int_\O (|\nabla v_1|^2+|\nabla z_1|^2)\,dx\leq
K_0|\lambda|^{-1}||U||_{\cal H}^2.
\end{align}
The application of Hahn-Banach Theorem shows that there exists a function $\psi$ in $H_0^2(\O)$ such that $|| \psi||_2=||\sqrt{a}v_2||_{-2}$ and
\begin{align}\label{ey2}
||\sqrt{a}v_2||_{-2}^2=\int_\O \sqrt{a}v_2\bar\psi\,dx\leq ||v_2||_{-2}|| \sqrt{a}\psi||_2\leq K_0|\lambda|^{-1} \left(||Z||_{\cal H}+||\lambda|^{-1}||U||_{\cal H}\right)|| \sqrt{a}\psi||_2.
\end{align}
Notice that with $a$ in $C^2(\bar\O)$ and the structural constraint \eqref{sa} being enforced, the function $\sqrt{a}\psi$ is in $H_0^2(\O)$. The proof of this claim is elementary and may be found in the Appendix of \cite[see (45) in Appendix A]{lasteb}.\\
 By Rellich inequality, we have 
\begin{align}\label{ey3}
|| \sqrt{a}\psi||_2^2&=\int_\O|\Delta( \sqrt{a}\psi)|^2\,dx\nag\\&=\int_\O|\psi(2^{-1}a^{-{1\over2}}\Delta a-4^{-1}a^{-{3\over2}}|\nabla a|^2)+a^{-{1\over2}}(\nabla a\cdot\nabla\psi)+\sqrt{a}\Delta\psi|^2\,dx
\end{align}The application of  the Cauchy-Schwarz inequality, Rellich inequality and \eqref{sa} lead to
\begin{align}\label{ey4}
|| \sqrt{a}\psi||_2\leq K_0|\Delta\psi|_2\end{align}so that \eqre{ey2} becomes
\begin{align}\label{ey21}
||\sqrt{a}v_2||_{-2}\leq K_0|\lambda|^{-1} \left(||Z||_{\cal H}+||\lambda|^{-1}||U||_{\cal H}\right).
\end{align}
In a similar way, one shows
\begin{align}\label{ey21b}
||\sqrt{a}z_2||_{-2}\leq K_0|\lambda|^{-1} \left(||Z||_{\cal H}+||\lambda|^{-1}||U||_{\cal H}\right).
\end{align}
Gathering \eqref{ez8}, \eqre{ez9}, \eqre{ey1}, \eqref{ey21} and \eqre{ey21b}, we derive
\begin{align}\label{ez80}
|\sqrt{a}(v+z)|_2&\leq K_0|\lambda|^{-1}||U||_{\cal H}\notag\\&\hskip.4cm+K_0|\lambda|^{-{1\over3}} \left(||Z||_{\cal H}+||\lambda|^{-1}||U||_{\cal H}\right)^{1\over3}\left(||Z||_{\cal H}+||Z||_{\cal H}^{1\over2}||U||_{\cal H}^{1\over2}+|\lambda|^{-{1\over2}}||U||_{\cal H}\right)^{2\over3}.
\end{align}
The last inequality simplifies to
\begin{align}\label{ez81}
|\sqrt{a}(v+z)|_2\leq K_0|\lambda|^{-{1\over3}}\left(  ||Z||_{\cal H}+||Z||_{\cal H}^{2\over3} ||U||_{\cal H}^{1\over3}+|\lambda|^{-{1\over3}}||Z||_{\cal H}^{1\over3}||U||_{\cal H}^{2\over3}+|\lambda|^{-{2\over3}}||U||_{\cal H}\right).
\end{align}
At this stage, we draw the reader's attention to the identity
\begin{equation}\label{ez81a}|\sqrt{a}v|_2^2+|\sqrt{a}z|_2^2=|\sqrt{a}(v+z)|_2^2-2\Re\int_\Omega{a}v\bar z\,dx.\end{equation}
To complete this step, it remains to estimate the last integral term. To do so, we shall use Green's formula. Multiply the second equation in \eqref{xae2} by $c\bar z)$ and the conjugate of its last equation by $dv$, then apply Green's formula over $\Omega$ to get the two identities
\begin{align}
c\int_\Omega av\bar z\,dx&=-\frac{cd}{i\lambda}\int_\Omega \Delta u(a\Delta\bar z+2\nabla a\cdot \nabla \bar z+\bar z\Delta a)\,dx\nag\\&
\hskip.2in+\frac{1}{i\lambda}\int_\Omega cag\bar z\,dx-\frac{c}{i\lambda}\int_\Omega a\nabla(v+z)\cdot(a\nabla\bar z+\bar z\nabla a)\,dx\label{ads1}\\~\nag\\
-d\int_\Omega a\bar zv\,dx&=-\frac{cd}{i\lambda}\int_\Omega \Delta \bar w(a\Delta v +2\nabla a\cdot \nabla v+v\Delta a)\,dx\nag\\&
\hskip.2in+\frac{1}{i\lambda}\int_\Omega da\bar \ell u\,dx-\frac{d}{i\lambda}\int_\Omega a\nabla(\bar v+\bar z)\cdot(a\nabla v+v\nabla a)\,dx\label{ads2}
\end{align}
Adding \eqref{ads2} to \eqref{ads1}, using the first and third equations in \eqref{xae2}, and taking real parts, one derives
\begin{align}\label{ads3}
\Re\int_\Omega av\bar z\,dx&=\frac{1}{c-d}\Bigg[\Re\frac{1}{i\lambda}\int_\Omega a\Big[cd(\Delta u\Delta\bar h-\Delta \bar w\Delta f) +cg\bar z+d\ell v\Big]\,dx\nag\\&\hskip.2in
-\Re\frac{cd}{i\lambda}\int_\Omega\Delta\bar w(2\nabla a\cdot \nabla v+v\Delta a)+\Delta u(2\nabla a\cdot \nabla \bar z+\bar z\Delta a)\,dx\\&\hskip.2in- \Re\frac{1}{i\lambda}\int_\Omega a\Big[c\nabla(v+z)\cdot(a\nabla\bar z+\bar z\nabla a)+d\nabla(\bar v+\bar z)\cdot(a\nabla v+v\nabla a)\Big]\,dx.\nag
\Bigg]\end{align}
The application of Cauchy-Schwarz inequality leads to the following estimates
\begin{align}\label{ads4}
\left|\Re\frac{1}{i\lambda}\int_\Omega a\Big[cd(\Delta u\Delta\bar h-\Delta \bar w\Delta f) +cg\bar z+d\ell v\Big]\,dx\right|&\leq K_0|\lambda|^{-1}(|\Delta u|_2|\Delta h|_2+|\Delta w|_2|\Delta f|_2)\nag\\&\hskip.2in+K_0|\lambda|^{-1}(|g|_2|z|_2+|v|_2|\ell|_2)\nag\\&\leq K_0|\lambda|^{-1}||Z||_{\cH}||U||_{\cH}.\end{align}
and 
\begin{align}\label{ads5}
&\left|\Re\frac{cd}{i\lambda}\int_\Omega\Delta\bar w(2\nabla a\cdot \nabla v+v\Delta a)+\Delta u(2\nabla a\cdot \nabla \bar z+\bar z\Delta a)\,dx\right|
\nag\\&\leq K_0|\lambda|^{-1}(|\Delta w|_2(|\nabla v |_2+|v|_2)+|\Delta u|_2(|\nabla z |_2+|z|_2))\nag\\&\leq K_0|\lambda|^{-1}||Z||_{\cH}^2+K_0|\lambda|^{-1}||Z||_{\cH}(|v|_2^\frac{1}{2}|\Delta v|_2^\frac{1}{2}+|z|_2^\frac{1}{2}|\Delta z|_2^\frac{1}{2}),
\end{align}where in the last inequality, we have invoked the elementary interpolation inequality: $|\nabla q|_2\leq C|q|_2^\frac{1}{2}|\Delta q|_2^\frac{1}{2}$ for every $q$ in $H_0^2(\Omega)$.\\ Using the first and third equations in \eqref{xae2}, one readily derives
\begin{equation}\label{ads5a}
|\Delta v|_2^\frac{1}{2}+|\Delta z|_2^\frac{1}{2}\leq K_0(|\lambda|^\frac{1}{2}||Z||_{\cH}^\frac{1}{2}+||U||_{\cH}^\frac{1}{2}).\end{equation}
The combination of \eqref{ads5} and \eqref{ads5a} yields
\begin{align}\label{ads5b}
&\left|\Re\frac{cd}{i\lambda}\int_\Omega\Delta\bar w(2\nabla a\cdot \nabla v+v\Delta a)+\Delta u(2\nabla a\cdot \nabla \bar z+\bar z\Delta a)\,dx\right|
\nag\\&\leq K_0\left(|\lambda|^{-\frac{1}{2}}||Z||_{\cH}^2+|\lambda|^{-1}||Z||_{\cH}^{\frac{3}{2}}||U||^\frac{1}{2}\right).\end{align}
Similarly, one derives
\begin{align}\label{ads6}
&\left|\Re\frac{1}{i\lambda}\int_\Omega a\Big[c\nabla(v+z)\cdot(a\nabla\bar z+\bar z\nabla a)+d\nabla(\bar v+\bar z)\cdot(a\nabla v+v\nabla a)\Big]\,dx\right|\nag\\&
\leq K_0|\lambda|^{-1}||Z||_{\cH}^{\frac{1}{2}}||U||_{\cH}^\frac{1}{2}\left(||Z||_{\cH}+|\lambda|^\frac{1}{2}||Z||_{\cH}+||Z||_{\cH}^\frac{1}{2}||U||_{\cH}^\frac{1}{2}\right)\\&
\leq K_0\left(|\lambda|^{-\frac{1}{2}}||U||_{\cH}^{\frac{1}{2}}||Z||_{\cH}^\frac{3}{2}+|\lambda|^{-1}||Z||_{\cH}||U||_{\cH}\right).\nag\end{align}
Gathering \eqref{ads3}, \eqref{ads4}, \eqref{ads5b} and \eqref{ads6}, we find
\begin{align}\label{ads7}
&\left|\Re\int_\Omega av\bar z\,dx\right|\leq  K_0\left(|\lambda|^{-\frac{1}{2}}||Z||_{\cH}^2+|\lambda|^{-\frac{1}{2}}||U||_{\cH}^{\frac{1}{2}}||Z||_{\cH}^\frac{3}{2}+|\lambda|^{-1}||Z||_{\cH}||U||_{\cH}\right).\end{align}
Finally, the combination of \eqref{ez81}, \eqref{ez81a} and \eqref{ads7} leads to the following estimate for the localized kinetic energy
\begin{align}\label{ads8}
|\sqrt{a}v|_2^2+|\sqrt{a}z|_2^2&\leq K_0|\lambda|^{-{2\over3}}\left(  ||Z||_{\cal H}^{4\over3} ||U||_{\cal H}^{2\over3}+|\lambda|^{-{2\over3}}||Z||_{\cal H}^{2\over3}||U||_{\cal H}^{4\over3}+|\lambda|^{-{4\over3}}||U||_{\cal H}^2\right)\nag\\&\hskip.2in+K_0\left(|\lambda|^{-\frac{1}{2}}||Z||_{\cH}^2+|\lambda|^{-\frac{1}{2}}||U||_{\cH}^{\frac{1}{2}}||Z||_{\cH}^\frac{3}{2}+|\lambda|^{-1}||Z||_{\cH}||U||_{\cH}\right) .
\end{align}

\nt
The above estimate is "optimal" in the sense that there is no "loss" of analyticity; the latter is totally transferred from the "parabolic" components $v_1$ and $z_1$. 
It gives  perfect scaling, \cite{lasteb}: 
\begin{align}\label{step1}
 |a^{1/2} v|_2 +|a^{1/2} z|_2&\sim  ||  \lambda^{-1}U||_{\cH} + ||  |\lambda|^{-1/4}Z||_{\cH} + ||Z||^{2/3}_{\cH} ||\lambda^{-1} U||^{1/3}_{\cH}   + || \lambda^{-1} U||^{2/3}_{\cH} ||Z||^{1/3}_{\cH}\nag\\&\hskip.2in
 +||Z||^{3/4}_{\cH} ||\lambda^{-1} U||^{1/4}_{\cH}   + || \lambda^{-1} U||^{1/2}_{\cH} ||Z||^{1/2}_{\cH}.
 \end{align}
\vskip.2cm\noindent
{\bf STEP 2: Estimation of localized potential energy.} First, we are going to use the equipartition of localized energy to estimate the localized potential energy of the $(u,v)-$system. The estimate of the localized potential energy of the $(w,z)-$system is established similarly; so the details will be omitted. Then, we will combine the two estimates to derive an estimate of the localized potential energy of the coupled system. Finally, we shall gather the estimates of localized kinetic and potential energies to derive an estimate for the localized energy of the coupled system.\\
Multiply  the second equation in \eqref{xae2}  by $ a\bar v$, use Green's formula and take real parts to derive
\begin{align}\label{xae6}
d\int_\Omega a|\Delta u|^2\,dx&=\int_\Omega a |v|^2\,dx+\Re\frac{d}{i\lambda}\int_\Omega\Delta u(2\nabla a\cdot \nabla\bar v+\bar v\Delta a-a\Delta\bar f)\,dx\nag\\&\hskip.2in \Re\frac{1}{i\lambda}\int_\Omega a\nabla(v+z)\cdot(a\nabla\bar v+\bar v\nabla a)\,dx-\Re\frac{1}{i\lambda}\int_\Omega ag\bar v\,dx.
\end{align}
Proceeding similarly, we find
\begin{align}\label{xae6r}
c\int_\Omega a|\Delta w|^2\,dx&=\int_\Omega a |z|^2\,dx+\Re\frac{c}{i\lambda}\int_\Omega\Delta w(2\nabla a\cdot \nabla\bar z+\bar z\Delta a-a\Delta\bar h)\,dx\nag\\&\hskip.2in \Re\frac{1}{i\lambda}\int_\Omega a\nabla(v+z)\cdot(a\nabla\bar z+\bar z\nabla a)\,dx-\Re\frac{1}{i\lambda}\int_\Omega a\ell\bar z\,dx.
\end{align}
Adding \eqref{xae6} to \eqref{xae6r}, we obtain
\begin{align}\label{xae6s}
\int_\Omega a(d|\Delta u|^2+c|\Delta w|^2)\,dx&=\int_\Omega a (|v|^2+|z|^2)\,dx+\Re\frac{d}{i\lambda}\int_\Omega\Delta u(2\nabla a\cdot \nabla\bar v+\bar v\Delta a-a\Delta\bar f)\,dx\nag\\&\hskip.2in \Re\frac{1}{i\lambda}\int_\Omega a\nabla(v+z)\cdot(\bar v\nabla a)\,dx-\Re\frac{1}{i\lambda}\int_\Omega ag\bar v\,dx\\&\hskip.4cm
+\Re\frac{c}{i\lambda}\int_\Omega\Delta w(2\nabla a\cdot \nabla\bar z+\bar z\Delta a-a\Delta\bar h)\,dx\nag\\&\hskip.2in \Re\frac{1}{i\lambda}\int_\Omega a\nabla(v+z)\cdot(\bar z\nabla a)\,dx-\Re\frac{1}{i\lambda}\int_\Omega a\ell\bar z\,dx,\nag
\end{align} where the imaginary term involving $|\nabla(v+z)|^2$ disappears.\\
 Now, we are going to estimate the crossed products coming from the $(u,v)$-system; similar estimates will hold for the corresponding terms coming from the $(w,z)$-system.\\
 Thanks to the Cauchy-Schwarz inequality and the dissipation law, one readily gets the following estimates
\begin{align}\label{ahte1}
\left|\Re\frac{1}{i\lambda}\int_\Omega d\Delta u(\bar v\Delta a-a\Delta\bar f)-g\bar v\,dx\right|&\leq K_0|\lambda|^{-1}(|\Delta u|_2(|v|_2+|\Delta f|_2)+|g|_2|v|_2)\nag\\&\leq K_0|\lambda|^{-1}(||Z||_{\cH}^2+||Z||_{\cH}||U||_{\cH})\\ \left|\Re\frac{1}{i\lambda}\int_\Omega a\nabla(v+z)\cdot(\bar v\nabla a)\,dx\right|&\leq K_0|\lambda|^{-1}||U||_{\cH}^{\frac{1}{2}}||Z||_{\cH}^\frac{3}{2}.\nag
\end{align}
Similarly, one derives, (by using the interpolation inequality: $|\nabla v|_2\leq K_0|v|_2^\frac{1}{2}|\Delta v|_2^\frac{1}{2}$ and the first equation in \eqref{xae2}):
\begin{align}\label{ahte2}
\left|\Re\frac{d}{i\lambda}\int_\Omega\Delta u(2\nabla a\cdot\nabla\bar v)\,dx\right|&\leq K_0|\lambda|^{-1}|\Delta u|_2|v|_2^{\frac{1}{2}}(|\lambda|^\frac{1}{2}|\Delta u|_2^\frac{1}{2}+|\Delta f|_2^\frac{1}{2})\nag\\&\leq K_0(|\lambda|^{-\frac{1}{2}}||Z||_{\cH}^2+|\lambda|^{-1}||U||_{\cH}^{\frac{1}{2}}||Z||_{\cH}^\frac{3}{2}),
\end{align}
as well as
\begin{align}\label{ahte1r}
\left|\Re\frac{1}{i\lambda}\int_\Omega c\Delta w(\bar z\Delta a-a\Delta\bar h)-\ell\bar z\,dx\right|&\leq K_0|\lambda|^{-1}(|\Delta w|_2(|z|_2+|\Delta h|_2)+|\ell|_2|z|_2)\nag\\&\leq K_0|\lambda|^{-1}(||Z||_{\cH}^2+||Z||_{\cH}||U||_{\cH})
\\ 
\left|\Re\frac{1}{i\lambda}\int_\Omega a\nabla(v+z)\cdot(\bar z\nabla a)\,dx\right|&\leq K_0|\lambda|^{-1}||U||_{\cH}^{\frac{1}{2}}||Z||_{\cH}^\frac{3}{2}.\nag
\end{align}
and
\begin{align}\label{ahte2r}
\left|\Re\frac{c}{i\lambda}\int_\Omega\Delta w(2\nabla a\cdot\nabla\bar z)\,dx\right|&\leq K_0|\lambda|^{-1}|\Delta w|_2|z|_2^{\frac{1}{2}}(|\lambda|^\frac{1}{2}|\Delta w|_2^\frac{1}{2}+|\Delta h|_2^\frac{1}{2})\nag\\&\leq K_0(|\lambda|^{-\frac{1}{2}}||Z||_{\cH}^2+|\lambda|^{-1}||U||_{\cH}^{\frac{1}{2}}||Z||_{\cH}^\frac{3}{2}),
\end{align}

The combination of \eqref{ads8} and \eqref{xae6s} to \eqref{ahte2r}, one finds
\begin{align}\label{ahte3}
&\int_\Omega a(|v|^2+d|\Delta u|^2)\,dx+\int_\Omega a(|z|^2+c|\Delta w|^2)\,dx\nag\\
&\leq K_0|\lambda|^{-{2\over3}}\left(  ||Z||_{\cal H}^{4\over3} ||U||_{\cal H}^{2\over3}+|\lambda|^{-{2\over3}}||Z||_{\cal H}^{2\over3}||U||_{\cal H}^{4\over3}+|\lambda|^{-{4\over3}}||U||_{\cal H}^2\right)\\&\hskip.2in+K_0\left(|\lambda|^{-\frac{1}{2}}||Z||_{\cH}^2+|\lambda|^{-1}||U||_{\cH}^{\frac{1}{2}}||Z||_{\cH}^\frac{3}{2}+|\lambda|^{-1}||Z||_{\cH}||U||_{\cH}\right) .\nag
\end{align}

Now that the localized energy has been estimated, we notice, as in the case of the localized kinetic energy that there is no loss in regularity; indeed, if the damping coefficient $a$ were a positive constant, the last two estimates would have led to the analyticity of the semigroup. In the remaining steps, we are going to propagate our "local regularity" to the whole domain. To do so requires the use of a first-order multiplier; this is where we are going to lose in regularity. By the way, thanks to \cite[Theorem 4.1]{liu-liu}, we know that localized structural or Kelvin-Voigt damping never leads to the analyticity of the semigroup; so, expecting a Gevrey-type regularity is the best regularity that one can hope for. 
\let\eqre=\eqref\let\o=\omega\let\d=\delta\let\a=\alpha\let\b=\beta
\vskip.2cm
\noindent
{\bf STEP 3: Geometric propagation of the damping effect.} In this portion of the proof, as indicated above, we will be using a first-order multiplier. To justify all computations that will follow, we find it useful to draw the reader's attention to the fact that we do have the requisite smoothness. Indeed, with the assumption on the domain, the damping coefficient, and the fact that $(u,w,v,z)$ lies in $D({\cal A})$, elliptic regularity shows that
$u,\,w\in H^4(\Omega)\cap H^2_0(\Omega)$.\\ 
Let $\a>0$ and $\b$ be real constants to be specified later. \\
Multiply  the conjugate of the second equation in \eqref{xae2}  by $ v$ and use its first equation as well as Green's formula to derive
\begin{align}\label{3e1}
-|v|_2^2+d|\Delta u|_2^2=\Re\frac{1}{i\lambda}\int_\Omega\{d\Delta\bar u\Delta f+\bar gv-a\nabla\overline{(v+z)}\cdot\nabla  v\}\,dx.
\end{align}
Similarly, one derives
\begin{align}\label{3e11}
-|z|_2^2+c|\Delta w|_2^2=\Re\frac{1}{i\lambda}\int_\Omega\{c\Delta\bar w\Delta h+\bar\ell z-a\nabla\overline{(v+z)}\cdot\nabla  z\}\,dx.
\end{align}
Adding those two equations and multiplying the resulting equation by $\b$, we find
\begin{align}\label{3e2}
&\beta\left(-(|v|_2^2+|z|_2^2)+d|\Delta u|_2^2+c|\Delta w|_2^2\right)\nag\\&=\beta\Re\frac{1}{i\lambda}\int_\Omega\{d\Delta\bar u\Delta f+\bar gv+c\Delta\bar w\Delta h+\bar\ell z\}\,dx,
\end{align}since the term involving $v+z$ disappears after the addition.\\
The application of the Cauchy-Schwarz inequality shows
\begin{align}\label{3e3}
&\left|\beta\Re\frac{1}{i\lambda}\int_\Omega\{d\Delta\bar u\Delta f+\bar gv+c\Delta\bar w\Delta h+\bar\ell z\}\,dx\right|\notag\\&\leq K_0|\lambda|^{-1}(|\Delta u|_2|\Delta f|_2+|g|_2|v|_2+|\Delta w|_2|\Delta h|_2+|\ell|_2|z|_2)\leq K_0|\lambda|^{-1}
||Z||_{\cH}||U||_{\cH}.
\end{align}
It then follows from \eqref{3e2} and \eqref{3e3}
\begin{align}\label{3e4}
&\beta\left(-(|v|_2^2+|z|_2^2)+d|\Delta u|_2^2+c|\Delta w|_2^2\right)\leq K_0|\lambda|^{-1}
||Z||_{\cH}||U||_{\cH}.
\end{align}
Now, we shall introduce some further notations that will be useful in the sequel. For each $j=1,\dots,J$,  where $J$ appears in the geometric constraint (GC) stated above, set $m^j(x)=x-x_0^j$. Let $0<\d_0<\d_1<\d$, where $\d$ is the one
given in (GC). Set
$$S=\displaystyle\left(\bigcup_{j=1}^J\G_j\right)\bigcup\left(\Omega\setminus\bigcup_{j=1}^J\Omega_j
\right),\quad
 Q_0={\cal N}_{\delta_0}(S),\quad Q_1={\cal N}_{\delta_1}(S),\quad\o_1=\O\cap Q_1,$$and
for each $j$, let
 $\varphi_j$ be a function  satisfying
 $$\varphi_j\in
W^{2,\infty}(\O),\quad 0\leq\varphi_j\leq1,\quad \var_j=1\hbox{ in }
\bar\Omega_j\setminus Q_1,\quad \var_j= 0\hbox{ in } \Omega\cap
Q_0.$$  

Now, multiply the second equation in \eqref{xae2} by $2\a \var_jm^j\cdot\nabla \bar u$, integrate over $\Omega_j$, take real parts and use the first equation in \eqref{xae2} to get
\begin{align}\label{3e4a}2\a\Re\int_{\Omega_j} g(\varphi_jm^j\cdot\nabla\bar u)\,dx&=2\a\Re\int_{\Omega_j} v\varphi_jm^j\cdot\nabla(-\bar v-\bar f)\,dx+2\a d\Re\int_{\Omega_j} \Delta^2u\varphi_jm^j\cdot\nabla\bar u\,dx\cr&
\hskip.4cm-2\a\Re\int_{\Omega_j}\hbox{div}(a\nabla( v+z))\cdot(\varphi_jm^j\cdot\nabla\bar u)\,dx.\end{align}
Proceeding similarly with the third and fourth equations in \eqref{xae2}, we obtain
\begin{align}\label{3e5}2\a\Re\int_{\Omega_j} \ell(\varphi_jm^j\cdot\nabla\bar w)\,dx&=2\a\Re\int_{\Omega_j} z\varphi_jm^j\cdot\nabla(-\bar z-\bar h)\,dx+2\a c\Re\int_{\Omega_j} \Delta^2w\varphi_jm^j\cdot\nabla\bar w\,dx\cr&
\hskip.4cm-2\a\Re\int_{\Omega_j}\hbox{div}(a\nabla( v+z))\cdot(\varphi_jm^j\cdot\nabla\bar w)\,dx.\end{align}
Now, we shall get some estimates for the $(u,v)$-system. Similar estimates hold for the $(w,z)$-system. Then we shall gather those estimates to get a first estimate of $||Z||_{\cH}$. \\ To this end, we are going to start with estimating the 
 term in the left hand side of \eqre{3e4}. Before proceeding with the actual estimate, we find it useful to rewrite that term by exploiting the decomposition of the velocity and the first equation in \eqre{xae2}, thereby getting
\begin{align}\label{ea1}2\a\Re\int_{\Omega_j} g(\varphi_jm^j\cdot\nabla\bar u)\,dx&=2\a\Re\int_{\Omega_j} (i\lambda v_1-\Delta v_1)(\varphi_jm^j\cdot\nabla\bar u)\,dx\notag\\&=2\a\Re\int_{\Omega_j} (- v_1\varphi_jm^j\cdot(\nabla\bar v+\nabla\bar f)-\Delta v_1(\varphi_jm^j\cdot\nabla\bar u)\,dx\\&=2\a\Re\int_{\Omega_j} (\bar v\text{div}(v_1\varphi_jm^j) -v_1\var_jm^j\cdot\nabla\bar f+\nabla v_1
\cdot\nabla(\varphi_jm^j\cdot\nabla\bar u)\,dx.\notag\end{align} Notice that  in the application of Green's formula, the boundary terms vanish. Indeed the boundary can be partioned into two parts so that on one part, $\var_j$ vanishes, and the other part is a subset of $\p\O$.\\ Before proceeding any further, we find it useful to recall the elementary identity: $|\Delta u|_2=|D^2u|_2$ for every $u$ in $H_0^2(\O)$, which, we shall implicitly use in the next estimate and later on.\\ Therefore, taking the sum over $j$ and applying the Cauchy-Schwarz and Poincar\'e  inequalities, yield, (keeping in mind that $|\lambda|>1$):
\begin{align}\label{ea2}2\a\Re\sum_{j=1}^J\int_{\Omega_j} g(\varphi_jm^j\cdot\nabla\bar u)\,dx&=2\a\Re\sum_{j=1}^J\int_{\Omega_j} (\bar v\text{div}(v_1\varphi_jm^j) -v_1\var_jm^j\cdot\nabla\bar f)\,dx\notag\\&\hskip.3cm+2\a\Re\sum_{j=1}^J\int_{\Omega_j}\nabla v_1
\cdot\nabla(\varphi_jm^j\cdot\nabla\bar u)\,dx\notag\\&\leq K_0(|v|_2+|\Delta u|_2)||v_1||_1+K_0|v_1|_2|\Delta f|_2\\&
\leq K_0\left(|\lambda|^{-{1\over2}}||Z||_{\cal H}||U||_{\cal H}+|\lambda|^{-1}||U||_{\cal H}^2\right)\nag\end{align}where in the last inequality, we use \eqre{vel1}.\\
 Next, we turn our attention to the terms in the right hand side of \eqre{3e4}. We are going to start with the term involving $f$, then we will examine the term involving $\Delta^2u$, afterwards, the term involving $\nabla \bar v$, and finally the term involving the damping.\\ To get a good estimate of the term involving $f$, first, we are going to use a duality argument between the functional space $H_0^1(\Omega)$ and its topological dual $H^{-1}(\Omega).$ To this end, notice that, (denoting the characteristic function of a set B by $1_B$):
 \begin{align}\label{velxb3}
-2\a\Re\int_{\Omega_j} v\varphi_jm^j\cdot\nabla\bar f\,dx&=-2\a\Re\int_{\Omega_j}v_1\varphi_jm^j\cdot\nabla\bar f\,dx-2\a\Re\langle v_2,(1_{\Omega_j}\varphi_jm^j\cdot\nabla\bar f)\rangle
\end{align}where $\langle.,.\rangle$ denotes the duality between $H^{-1}(\Omega)$ and $H_0^1(\Omega)$.\\ It then readily follows from Cauchy-Schwarz inequality, Poincar\'e type inequality and \eqref{vel1}:
 \begin{align}\label{velxb4}\left|2\a\Re\int_{\Omega_j}v_1\varphi_jm^j\cdot\nabla\bar f\,dx\right|\leq K_0|v_1|_2|\Delta f|_2\leq K_0|\lambda|^{-1}||U||_{\cal H}^2.\end{align}
and, thanks to \eqre{velxa3}: 
 \begin{align}\label{velxb5}\left|2\a\Re \langle v_2,(1_{\Omega_j}\varphi_jm^j\cdot\nabla\bar f)\rangle\right|&\leq ||v_2||_{-1}||1_{\Omega_j}\varphi_jm^j\cdot\nabla\bar f ||_1\nag\\&
\leq K_0|\lambda|^{-{1\over2}}(||Z||_{\cal H}+|\lambda|^{-1}||U||_{\cal H})|\Delta f|_2\\&
\leq K_0|\lambda|^{-{1\over2}}(||Z||_{\cal H}||U||_{\cal H}+|\lambda|^{-1}||U||_{\cal H}^2).\nag
\end{align}
Hence
 \begin{align}\label{velxb6}
\left|2\a\Re\sum_{j=1}^J\int_{\Omega_j}v\varphi_jm^j\cdot\nabla\bar f\,dx\right|\leq  K_0(|\lambda|^{-{1\over2}}||Z||_{\cal H}||U||_{\cal H}+|\lambda|^{-1}||U||_{\cal H}^2).
\end{align}
Applying Green's formula to the term involving $\Delta^2u$, we find
\begin{align}\label{ec}&2\a d\Re\int_{\Omega_j}\Delta^2 u(\varphi_jm^j\nabla\bar u)\,dx\notag\\&=-2\a d\Re\int_{\Omega_j}\{\nabla\Delta u \cdot(\nabla\varphi_j)m^j\cdot\nabla\bar u+\var_j\nabla\Delta u\cdot\nabla\bar u\}\,dx\\&\hskip.4cm -2\a d\Re\int_{\Omega_j}\var_jD^2\bar u(m^j,\nabla\Delta u)\,dx+2\a d\int_{\p\Omega_j}\varphi_j\p_{\nu^j}\Delta u m^j\cdot\nabla\bar u\,d\Gamma.\notag\end{align}
Applying Green's formula once more, we get
\begin{align}\label{ecy}&2\a d\Re\int_{\Omega_j}\Delta^2 u(\varphi_jm^j\nabla\bar u)\,dx\notag
\\&=2\a d\int_\O\{\Delta u\Delta\var_j m^j\cdot\nabla\bar u+\Delta u\nabla\var_j\cdot\nabla\bar u+\Delta uD^2\bar u(\nabla\var_j,m^j)\}\,dx\notag\\&\hskip.4cm -2\a d\Re\int_{\p\Omega_j}\Delta u \nu^j\cdot\nabla\varphi_jm^j\cdot\nabla\bar u\,d\Gamma
+2\a d\Re\int_{\Omega_j}\Delta u \nabla\varphi_j\cdot\nabla\bar u\,dx \\&\hskip.4cm+4\a d\int_{\Omega_j}\varphi_j|\Delta u|^2\,dx-2\a d\int_{\p\Omega_j}\varphi_j\Delta u\p_{\nu^j}\bar u\,d\Gamma\notag\\&\hskip.4cm+2\a d\Re\int_{\Omega_j}\Delta uD^2\bar u(\nabla\var_j,m^j)\,dx+\a d\int_{\Omega_j}\var_jm^j\cdot\nabla(|\Delta u|^2)\,dx\notag\\&\hskip.4cm-2\a d\Re\int_{\p\Omega_j}\varphi_j\Delta uD^2\bar u(m^j,\nu^j)\,d\Gamma+2\a d\int_{\p\Omega_j}\varphi_j\p_{\nu^j}\Delta u m^j\cdot\nabla\bar u\,d\Gamma,\notag\end{align} where $D^2u=(\p_k\p_\ell u;k,\,\ell=1,\,2,\,...,\,N)$, and for every $x$ in $\Omega$, the notation $D^2u(x)(p,q)$ stands for the sum $\displaystyle\sum_{k,\ell=_1}^N\p_k\p_\ell u(x)p_\ell q_k$ for all vectors $p,~q$ in ${\mathbb R}^N$.\\
Now, applying Green's formula one more time, we find
\begin{align}\label{ec0}\a d\Re\int_{\Omega_j}\var_jm^j\cdot\nabla(|\Delta u|^2)\,dx&=-N\a d\int_{\Omega_j}\var_j|\Delta u|^2\,dx-\a d\int_{\Omega_j}m^j\cdot\nabla\var_j|\Delta u|^2\,dx\notag\\&\hskip.4cm+\a d\int_{\p\Omega_j}\var_jm^j\cdot\nu^j|\Delta u|^2\,d\Gamma.\end{align}
Notice that if as in \cite{lu}, we set for each $j$, $S_j=\G_j\cup (\p\Omega_j\cap\O)$,
then one checks that $\var_j=0$ on $S_j$. Furthermore, we have
$\tilde S_j:=\p\Omega_j\setminus S_j\subset\G_j^c\cap\p\O$, ($A^c$ denotes the
complement of $A$); consequently, for each $j$, all boundary terms in \eqref{ecy} vanish, except for the boundary term involving $D^2\bar u$. However, that term also vanishes on $S_j$. Thus, with this observation in mind and \eqref{ec0}, the identity \eqref{ecy} becomes
\begin{align}\label{ec1}&2\a d\Re\int_{\Omega_j}\Delta^2 u(\varphi_jm^j\nabla\bar u)\,dx\notag\\&=2\a d\Re\int_\O\{\Delta u\Delta\var_j m^j\cdot\nabla\bar u+\Delta u\nabla\var_j\cdot\nabla\bar u\}\,dx\notag\\&\hskip.4cm +2\a d\Re\int_\O\Delta uD^2\bar u(m^j,\nabla\var_j)\,dx+2\a d\Re\int_{\Omega_j}\Delta u \nabla\varphi_j\cdot\nabla\bar u\,dx\notag\\&\hskip.4cm 
 -(N-4)\a d\int_{\Omega_j}(\varphi_j-1)|\Delta u|^2\,dx+2\a d\Re\int_{\Omega_j}\Delta uD^2\bar u(\nabla\var_j,m^j)\,dx\\&\hskip.4cm-(N-4)\a d\int_{\Omega_j}|\Delta u|^2\,dx-\a d\int_{\Omega_j}m^j\cdot\nabla\var_j|\Delta u|^2\,dx-\a d\int_{\tilde S_j}\varphi_j|\Delta u|^2m^j\cdot\nu^j\,d\Gamma,\notag\end{align} where the boundary integral is obtained by noticing that $\p_i\p_ju=\nu^i\nu^j\Delta u$ on $\p\O$. Also notice that
$$-\a d\int_{\tilde S_j}\varphi_j|\Delta u|^2m^j\cdot\nu^j\,d\Gamma\geq0$$ since $\tilde S_j\subset\Gamma_j^c$.\\ 
Taking the sum over $j$, then  applying the Cauchy-Schwarz inequality, we find
\begin{align}\label{ec2}&2\a d\Re\sum_{j=1}^J\int_{\Omega_j}\Delta^2 u(\varphi_jm^j\nabla\bar u)\,dx\\&\geq -K_0|\Delta u|_2\left(\int_{\o_1}|\nabla u|^2\,dx\right)^{1\over2}-K_0|\Delta u|_2\left(\int_{\o_1}|\Delta u|^2\,dx\right)^{1\over2}\notag\\&\hskip.4cm\displaystyle -K_0\int_{\o_1}|\Delta u|^2\,dx-(N-4)\alpha d|\Delta u|_2^2
\notag\\&\hskip.4cm+(N-4)\alpha d\int_{\O\setminus\bigcup_{j=1}^J\Omega_j}|\Delta u|^2\,dx,\notag\end{align}which simplifies to
\begin{align}\label{ecxa}&2\a d\Re\sum_{j=1}^J\int_{\Omega_j}\Delta^2 u(\varphi_jm^j\nabla\bar u)\,dx\notag\\&
\geq -K_0|\Delta u|_2\left(\int_{\o_1}|\nabla u|^2\,dx\right)^{1\over2}-K_0|\Delta u|_2\left(\int_{\o_1}|\Delta u|^2\,dx\right)^{1\over2}\\&\hskip.4cm\displaystyle -K_0\int_{\o_1}|\Delta u|^2\,dx-(N-4)\alpha d|\Delta u|_2^2,\notag
\end{align}where the last inequality is obtained by absorbing the last term in \eqref{ec2} into\\ $-K_0\int_{\o_1}|\Delta u|^2\,dx$ by recalling that $\O\setminus\bigcup_{j=1}^J\Omega_j\subset\o_1$.\\
The application of Green's formula shows
\begin{align}\label{eb}-2\a\Re\int_{\Omega_j} v\varphi_jm^j\cdot\nabla\bar v\,dx&=-\a\int_{\Omega_j} \varphi_jm^j\cdot\nabla( |v|^2)\,dx\nag\\&=\a N\int_{\Omega_j}\varphi_j|v|^2\,dx+\a\int_{\Omega_j}(m^j\cdot\nabla\varphi_j)|v|^2\,dx.\end{align}where the boundary integral vanishes as explained above.\\ Consequently, and thanks to the Cauchy-Schwarz inequality, we derive, after taking the sum over $j$:
\begin{align}\label{eb1}-2\a\Re\sum_{j=1}^J\int_{\Omega_j} v\varphi_jm^j\cdot\nabla\bar v\,dx&=\a N\sum_{j=1}^J\int_{\Omega_j}(\varphi_j-1)|v|^2\,dx+\a N\sum_{j=1}^J\int_{\Omega_j}|v|^2\,dx\notag\\&\hskip.4cm+\a\sum_{j=1}^J\int_{\Omega_j}(m^j\cdot\nabla\varphi_j)|v|^2\,dx\\&\geq -K_0\int_{\o_1}|v|^2\,dx+\a N|v|_2^2-\a N\int_{\O\setminus\bigcup_{j=1}^J\Omega_j}|v|^2\,dx\notag\\&\geq -K_0\int_{\O}a|v|^2\,dx+\a N|v|_2^2\notag\end{align}by taking into account that both $\var_j-1$ and $\nabla\var_j$ are nonzero in $\o_1$ only, and the damping coefficient  $d$ is bounded from below by a positive constant in $\o_1$.\\
For the remaining term in \eqre{3e4}, proceeding similarly, one derives, (notice that we introduce w here; this will ensure that our ultimate estimate involves $D^2(u+w)$ instead of $D^2u$ and $D^2w$ separately, which leads to a worse regularity estimate than the one desired):
\begin{align}\label{ez1}
&-2\a\Re\sum_{j=1}^J\int_{\Omega_j}\hbox{div}(a\nabla( v+z))(\varphi_jm^j\cdot\nabla(\bar u+\bar w))\,dx+2\a\Re\sum_{j=1}^J\int_{\Omega_j}\hbox{div}(a\nabla( v+z))(\varphi_jm^j\cdot\nabla\bar w)\,dx\nag\\&=2\a\Re\sum_{j=1}^J\int_{\Omega_j}a(\nabla (v+z)\cdot\nabla\varphi_j)(m^j\cdot\nabla(\bar u+\bar w))\,dx+2\a\Re\sum_{j=1}^J\int_{\Omega_j}a\var_j\nabla( v+z)\cdot\nabla(\bar u+\bar w))\,dx\nag\\&\hskip.4cm+2\a\Re\sum_{j=1}^J\int_{\Omega_j}a\var_j D^2(\bar u+\bar w)(\nabla( v+z),m^j)\,dx+2\a\Re\sum_{j=1}^J\int_{\Omega_j}\hbox{div}(a\nabla( v+z))(\varphi_jm^j\cdot\nabla\bar w)\,dx\nag\\&\geq -K_0\left(\int_{\O}a|\nabla( v+z)|^2\,dx\right)^{1\over2}\left(\int_{\O}a|\nabla( u+w)|^2\,dx\right)^{1\over2}+2\a\Re\sum_{j=1}^J\int_{\Omega_j}\hbox{div}(a\nabla( v+z))(\varphi_jm^j\cdot\nabla\bar w)\,dx\\&\hskip.4cm-K_0\left(\int_{\O}a|\nabla( v+z)|^2\,dx\right)^{1\over2}\left(\int_{\O}a|D^2(u+w)|^2\,dx\right)^{1\over2}.\nag\end{align}
Consequently, invoking \eqref{dislaw1}, we get
\begin{align}\label{ez1r}
&-2\a\Re\sum_{j=1}^J\int_{\Omega_j}\hbox{div}(a\nabla( v+z))(\varphi_jm^j\cdot\nabla(\bar u+\bar w))\,dx+2\a\Re\sum_{j=1}^J\int_{\Omega_j}\hbox{div}(a\nabla( v+z))(\varphi_jm^j\cdot\nabla\bar w)\,dx\nag\\&\geq
 -K_0|\lambda|^{-1}\left(||U||_{\cal H}^\frac{3}{2}||Z||_{\cal H}^\frac{1}{2}+||U||_{\cal H}||Z||_{\cal H}\right)-K_0||U||_{\cal H}^{1\over2}||Z||_{\cal H}^{1\over2}\left(\int_{\O}a|D^2(u+w)|^2\,dx\right)^{1\over2}\nag\\&\hskip.4cm+2\a\Re\sum_{j=1}^J\int_{\Omega_j}\hbox{div}(a\nabla( v+z))(\varphi_jm^j\cdot\nabla\bar w)\,dx.\end{align}
 Combining \eqref{3e4a}, \eqre{ea2}, \eqre{velxb6}, \eqre{ecxa}, \eqre{eb1} and \eqre{ez1r}, we get
\begin{align}\label{ez2}
\a N|v|_2^2-(N-4)\a d|\Delta u|_2^2&\leq K_0\left(|\lambda|^{-{1\over2}}||Z||_{\cal H}||U||_{\cal H}+|\lambda|^{-1}||U||_{\cal H}^2\right)+K_0\int_{\O}a|\Delta u|^2\,dx
\nag\\&\hskip.4cm+
K_0|\Delta u|_2|\nabla u|_2+ K_0|\Delta u|_2\left(\int_{\O}a|\Delta u|^2\,dx\right)^{1\over2}+K_0\int_{\O}a|v|^2\,dx\\&\hskip.4cm+K_0|\lambda|^{-1}||U||_{\cal H}^{3\over2}||Z||_{\cal H}^{1\over2}+
K_0||U||_{\cal H}^{1\over2}||Z||_{\cal H}^{1\over2}\left(\int_{\O}a|D^2(u+w)|^2\,dx\right)^{1\over2}\nag
\\&\hskip.4cm-2\a\Re\sum_{j=1}^J\int_{\Omega_j}\hbox{div}(a\nabla( v+z))(\varphi_jm^j\cdot\nabla\bar w)\,dx.\nag
\end{align}
Proceeding the same way for the $(w,z)$-system, estimating all terms except for the term involving the damping, we derive
\begin{align}\label{ez2b}
\a N|z|_2^2-(N-4)\a c|\Delta w|_2^2&\leq K_0\left(|\lambda|^{-{1\over2}}||Z||_{\cal H}||U||_{\cal H}+|\lambda|^{-1}||U||_{\cal H}^2\right)+K_0\int_{\O}a|\Delta w|^2\,dx
\nag\\&\hskip.4cm+
K_0|\Delta w|_2|\nabla w|_2+|\Delta w|_2\left(\int_{\O}a|\Delta w|^2\,dx\right)^{1\over2}+K_0\int_{\O}a|z|^2\,dx\\&\hskip.4cm+2\a\Re\sum_{j=1}^J\int_{\Omega_j}\hbox{div}(a\nabla( v+z))(\varphi_jm^j\cdot\nabla\bar w)\,dx\nag
\end{align}
Gathering \eqre{3e4},\eqref{ez2} and \eqre{ez2b}, then choosing the constants $\a$ and $\beta$ with \begin{align}\label{paramcond}\a(N-4)<\b<\a N\end{align} we find, (notice that the factors of both the kinetic and potential energies are positive thanks to this condition on the constants $\alpha$ and $\beta$):
\begin{align}\label{ez3}
&(\a N-\b)(|v|_2^2+|z|_2^2)+(\b-(N-4)\a) (d|\Delta u|_2^2+c|\Delta w|_2^2)\nag\\&\leq K_0\left(|\lambda|^{-{1\over2}}||Z||_{\cal H}||U||_{\cal H}+|\lambda|^{-1}||U||_{\cal H}^2\right)+K_0\int_{\O}a(|\Delta u|^2+|\Delta w|^2)\,dx
\nag\\&\hskip.4cm+
K_0|\Delta u|_2|\nabla u|_2+|\Delta u|_2\left(\int_{\O}a|\Delta u|^2\,dx\right)^{1\over2}+K_0\int_{\O}a(|v|^2+|z|^2)\,dx\\&\hskip.4cm+
K_0|\Delta w|_2|\nabla w|_2+|\Delta w|_2\left(\int_{\O}a|\Delta w|^2\,dx\right)^{1\over2}\nag\\&\hskip.4cm+K_0|\lambda|^{-1}||U||_{\cal H}^{3\over2}||Z||_{\cal H}^{1\over2}+
K_0||U||_{\cal H}^{1\over2}||Z||_{\cal H}^{1\over2}\left(\int_{\O}a|D^2(u+w)|^2\,dx\right)^{1\over2}\nag
\end{align}
At this stage, we recall the classical interpolation inequality as well as the first equation in \eqref{xae2} , it follows
\begin{align}\label{eza3}
|\nabla u|_2^2\leq K_0|u|_2|\Delta u|_2,
\end{align}from which one derives, thanks to the first equation in \eqref{xae2}
\begin{align}\label{ezb3}
|\nabla u|_2^2\leq K_0|\lambda|^{-1}(||Z||_\cH^2+||Z||_\cH||U||_\cH).\end{align}
Similarly, one derives
\begin{align}\label{ezb3b}
|\nabla w|_2^2\leq K_0|\lambda|^{-1}(||Z||_\cH^2+||Z||_\cH||U||_\cH).\end{align}
Using Young inequality and \eqre{ezb3}-\eqref{ezb3b}, it readily follows
\begin{align}\label{ez4}
&K_0|\Delta u|_2|\nabla u|_2+|\Delta u|_2\left(\int_{\O}a|\Delta u|^2\,dx\right)^{1\over2}+K_0|\Delta w|_2|\nabla w|_2+|\Delta w|_2\left(\int_{\O}a|\Delta w|^2\,dx\right)^{1\over2}\nag\\&\leq \frac{(\b-(N-4)\a) }{2}(d|\Delta u|_2^2+c|\Delta w|_2^2)+K_0(|\nabla u|_2^2+|\nabla w|_2^2)+K_0\int_{\O}a(|\Delta u|^2+|\Delta w|^2)\,dx
\\&\leq \frac{(\b-(N-4)\a) }{2}(d|\Delta u|_2^2+c|\Delta w|_2^2)+K_0|\lambda|^{-1}(||Z||_\cH^2+||Z||_\cH||U||_\cH)+K_0\int_{\O}a(|\Delta u|^2+|\Delta w|^2).\nag\end{align}
The combination of \eqref{ez3} and \eqref{ez4} leads to
\begin{align}\label{ezc3}
&|v|_2^2+|z|_2^2+ d|\Delta u|_2^2+c|\Delta w|_2^2\nag\\&\leq K_0\left(|\lambda|^{-{1\over2}}||Z||_{\cal H}||U||_{\cal H}+|\lambda|^{-1}||U||_{\cal H}^2\right)+K_0\int_{\O}a(|\Delta u|^2+|\Delta w|^2)\,dx
\nag\\&\hskip.4cm+
K_0|\lambda|^{-1}(||Z||_\cH^2+||Z||_\cH||U||_\cH)+K_0\int_{\O}a(|v|^2+|z|^2)\,dx\\&\hskip.4cm+K_0|\lambda|^{-1}||U||_{\cal H}^{3\over2}||Z||_{\cal H}^{1\over2}+
K_0||U||_{\cal H}^{1\over2}||Z||_{\cal H}^{1\over2}\left(\int_{\O}a(|D^2(u+w)|^2)\,dx\right)^{1\over2}\nag
\end{align}
It then follows from \eqre{ezc3}, (keeping in mind that $|\lambda|>1$):
\begin{align}\label{ez5}
||Z||_{\cal H}^2&\leq K_0\left(|\lambda|^{-{1\over2}}||Z||_{\cal H}||U||_{\cal H}+|\lambda|^{-1}||U||_{\cal H}^2\right)+K_0\int_{\O}a(d|\Delta u|^2+c|\Delta w|^2)\,dx
\nag\\&\hskip.4cm+
K_0\int_{\O}a(|v|^2+|z|^2)\,dx+K_0|\lambda|^{-1}||U||_{\cal H}^{3\over2}||Z||_{\cal H}^{1\over2}\\&\hskip.4cm+
K_0||U||_{\cal H}^{1\over2}||Z||_{\cal H}^{1\over2}\left(\int_{\O}a(|D^2(u+w)|^2)\,dx\right)^{1\over2}\nag
\end{align}
Invoking the localized energy estimates established in Step 2, (see \eqref{ahte3}), we obtain
\begin{align}\label{eza5}
||Z||_{\cal H}^2&\leq K_0\left(|\lambda|^{-{1\over2}}||Z||_{\cal H}||U||_{\cal H}+|\lambda|^{-1}||U||_{\cal H}^2\right)+K_0|\lambda|^{-1}||U||_{\cal H}^{3\over2}||Z||_{\cal H}^{1\over2}
\nag\\&\hskip.4cm+K_0|\lambda|^{-{2\over3}}\left(  ||Z||_{\cal H}^{4\over3} ||U||_{\cal H}^{2\over3}+|\lambda|^{-{2\over3}}||Z||_{\cal H}^{2\over3}||U||_{\cal H}^{4\over3}+|\lambda|^{-{4\over3}}||U||_{\cal H}^2\right)\\&\hskip.2in+K_0|\lambda|^{-\frac{1}{2}}||U||_{\cH}^{\frac{1}{2}}||Z||_{\cH}^\frac{3}{2}+
K_0||U||_{\cal H}^{1\over2}||Z||_{\cal H}^{1\over2}\left(\int_{\O}a(|D^2(u+w)|^2)\,dx\right)^{1\over2}\nag
\end{align}
~~\\
\let\v=\varepsilon
{\bf STEP 4: Ultimate estimate and completion of the proof.} In this step, we are going to estimate the integral term in \eqref{eza5}. We shall start with the term involving $D^2u$.\\ Applying Green's formula, we get, (using the Einstein summation convention on repeated indices):
\begin{align}\label{ez6}
\int_{\O}a|D^2(u+w)|^2\,dx&=\int_{\O}a\p_{k\ell}(u+w)\p_{k\ell}(\bar u+\bar w)\,dx\nag\\&=-\int_{\O}\p_\ell a\p_{k}(u+w)\p_{k\ell}(\bar u+\bar w)\,dx-\int_{\O}a\p_{k}(u+w)\p_{k}\Delta(\bar u+\bar w)\,dx\nag\\&=-\int_{\O}D^2(\bar u+\bar w)(\nabla a,\nabla (u+w))\,dx+\int_{\O}(\nabla a\cdot\nabla (u+w))\Delta(\bar u+\bar w)\,dx\nag\\&\hskip.4cm+\int_{\O}a|\Delta (u+w)|^2\,dx
\end{align}
Applying the Cauchy-Schwarz inequality and using \eqref{ezb3} as well as the fact that $|\nabla a|^2/a$ is bounded, it follows
\begin{align}\label{ez7}
\int_{\O}a|D^2(u+w)|^2\,dx&\leq K_0(|\Delta u|_2+|\Delta w|_2)\left(\int_{\O}a|\nabla (u+w)|^2\,dx\right)^{1\over2}+\int_{\O}a(|\Delta u|^2+|\Delta w|^2)\,dx\,dx\nag\\&
\leq K_0|\lambda|^{-1}\left(||U||_{\cH}^{\frac{1}{2}}||Z||_{\cH}^\frac{3}{2}+||Z||_{\cal H}||U||_{\cal H}\right)+\int_{\O}a(d|\Delta u|^2+c|\Delta w|^2)\,dx
\end{align}where we also used the identity: $|D^2u|_2=|\Delta u|_2$ as well as \eqref{dislaw1}.\\
Consequently,  \eqre{eza5} becomes
\begin{align}\label{ezb5}
||Z||_{\cal H}^2&\leq K_0\left(|\lambda|^{-{1\over2}}||Z||_{\cal H}||U||_{\cal H}+|\lambda|^{-1}||U||_{\cal H}^2\right)+K_0|\lambda|^{-1}||U||_{\cal H}^{3\over2}||Z||_{\cal H}^{1\over2}
\nag\\&\hskip.4cm+K_0|\lambda|^{-{2\over3}}\left(  ||Z||_{\cal H}^{4\over3} ||U||_{\cal H}^{2\over3}+|\lambda|^{-{2\over3}}||Z||_{\cal H}^{2\over3}||U||_{\cal H}^{4\over3}+|\lambda|^{-{4\over3}}||U||_{\cal H}^2\right)\\&\hskip.2in+K_0|\lambda|^{-1}||U||_{\cH}^{\frac{1}{2}}||Z||_{\cH}^\frac{3}{2}+K_0|\lambda|^{-\frac{1}{2}}||U||_{\cH}^{\frac{3}{4}}||Z||_{\cH}^\frac{5}{4}\nag\\&\hskip.4cm+
K_0||U||_{\cal H}^{1\over2}||Z||_{\cal H}^{1\over2}\left(\int_{\O}a(d|\Delta u|^2+c|\Delta w|^2)\,dx\right)^{1\over2}.\nag
\end{align}
Invoking \eqref{ahte3} once more, we get the following estimate
\begin{align}\label{ez55}
||Z||_{\cal H}^2&\leq K_0\left(|\lambda|^{-{1\over2}}||Z||_{\cal H}||U||_{\cal H}+|\lambda|^{-1}||U||_{\cal H}^2\right)+K_0|\lambda|^{-1}||U||_{\cal H}^{3\over2}||Z||_{\cal H}^{1\over2}
\nag\\&\hskip.4cm+K_0|\lambda|^{-{2\over3}}\left(  ||Z||_{\cal H}^{4\over3} ||U||_{\cal H}^{2\over3}+|\lambda|^{-{2\over3}}||Z||_{\cal H}^{2\over3}||U||_{\cal H}^{4\over3}\right)\nag\\&\hskip.2in+K_0|\lambda|^{-1}||U||_{\cH}^{\frac{1}{2}}||Z||_{\cH}^\frac{3}{2}+K_0|\lambda|^{-\frac{1}{2}}||U||_{\cH}^{\frac{3}{4}}||Z||_{\cH}^\frac{5}{4}\\&\hskip.4cm+
K_0||U||_{\cal H}^{1\over2}||Z||_{\cal H}^{1\over2}|\lambda|^{-{1\over3}}\left(  ||Z||_{\cal H}^{4\over3} ||U||_{\cal H}^{2\over3}+|\lambda|^{-{2\over3}}||Z||_{\cal H}^{2\over3}||U||_{\cal H}^{4\over3}\right)^\frac{1}{2}\nag\\&\hskip.4cm+
K_0||U||_{\cal H}^{1\over2}||Z||_{\cal H}^{1\over2}K_0\left(|\lambda|^{-\frac{1}{2}}||Z||_{\cH}^2+|\lambda|^{-1}||U||_{\cH}^{\frac{1}{2}}||Z||_{\cH}^\frac{3}{2}+|\lambda|^{-1}||Z||_{\cH}||U||_{\cH}\right)^\frac{1}{2}.\nag
\end{align}

One may rewrite this estimate as
\begin{align}\label{ez55r}
||Z||_{\cal H}^2&\leq K_0\left(||Z||_{\cal H}|||\lambda|^{-{1\over2}}U||_{\cal H}+|||\lambda|^{-{1\over2}}U||_{\cal H}^2\right)+K_0|||\lambda|^{-{2\over3}}U||_{\cal H}^{3\over2}||Z||_{\cal H}^{1\over2}
\nag\\&\hskip.4cm+K_0\left(  ||Z||_{\cal H}^{4\over3} |||\lambda|^{-1}U||_{\cal H}^{2\over3}+||Z||_{\cal H}^{2\over3}|||\lambda|^{-1}U||_{\cal H}^{4\over3}\right)\nag\\&\hskip.2in+K_0|||\lambda|^{-2}U||_{\cH}^{\frac{1}{2}}||Z||_{\cH}^\frac{3}{2}+K_0|||\lambda|^{-{2\over3}}U||_{\cH}^{\frac{3}{4}}||Z||_{\cH}^\frac{5}{4}\\&\hskip.4cm+
K_0\left(||Z||_{\cal H}^{7\over6} |||\lambda|^{-{2\over5}}U||_{\cal H}^{5\over6}+||Z||_{\cal H}^{5\over6}|||\lambda|^{-{4\over7}}U||_{\cal H}^{7\over6}\right)\nag\\&\hskip.4cm+
K_0\left( |||\lambda|^{-1}U||_{\cal H}^{1\over2}||Z||_{\cal H}^{3\over2}+|||\lambda|^{-{2\over3}}U||_{\cal H}^{3\over4}||Z||_{\cal H}^{5\over4}\right).\nag
\end{align}

Using Young inequality, one derives the desired estimate from \eqre{ez55r}, thereby completing the proof of Theorem \ref{reg}.\qed

\section{Proof of Theorem \ref{expstab}}\label{pexpstab}
This proof will rely on the following 
 semigroup exponential stability criterion:
 \begin{lem}\label{expdec}  {\bf (\cite[Theorem 3]{hf},  \cite[Corollary 4]{pr})} {\sl
Let ${\cal A}$ be the generator of a bounded $C_0$ semigroup
$(S(t))_{t\geq0}$ on a Hilbert space ${\cal H}$. Then
$(S(t))_{t\geq0}$ is exponentially stable if and only if:\par {\tt
i)} $i{\bb R}\subset\rho({\cal A})$, and\par  {\tt ii)}
$\sup\{||(i\lambda-{\cal A})^{-1}||_{{\cal L}({\cal H})};~\lambda\in {\bb R}\}<\infty$, where
$\rho({\cal A})$ denotes the resolvent set of} ${\cal A}$.\end{lem}
The resolvent set condition was established in Section 2. Therefore, it remains to prove the resolvent estimate. Thus, the proof will be completed as soon as we establish the following resolvent 
estimate \begin{equation}\label{res2}\exists K_0>0:\|(i\lambda{\cal I}-{\cal A})^{-1}U\|_{{\cal L}({\cal
H})}\leq K_0||U||_\cH,\,\forall U\in \cH,\,\forall\lambda\in{\mathbb R}.\end{equation}
To prove \eqref{res2}, it is enough to prove it for $|\lambda|\geq\lambda_0$ for some $\lambda_0>0$, since we can invoke the continuity of the resolvent on $|i\lambda|\leq\lambda_0$ to complete the proof.\\  For this purpose, let $U\in {\cal H}$, and let $\lambda$ be a real number with $|\lambda|>1$.
Since the range of $i\lambda{\cal I}-{\cal A}$ is ${\cal H}$, there
exists $Z\in{ D}({\cal A})$ such that
\begin{equation}\label{yae0}i\lambda Z-{\cal A} Z=U.\end{equation} The desired resolvent estimate will be established once we prove
\begin{equation}\label{yae1}\|Z\|_{\cal H}\leq K_0\,\|U\|_{\cal H},\end{equation}
The proof of this resolvent estimate is simpler than the one established in the preceding section; indeed, we will borrow elements from Step 3 of the proof in that section that do not require any smoothness of the damping coefficient $a$. Thus, we already have the following estimate:
\begin{align}\label{yae2}
&|v|_2^2+|z|_2^2+ d|\Delta u|_2^2+c|\Delta w|_2^2\nag\\&\leq K_0\sum_{j=1}^J\int_{\Omega_j} |g(\varphi_jm^j\cdot\nabla\bar u|+|\ell(\varphi_jm^j\cdot\nabla\bar w)|\,dx+K_0\int_{\omega_1}(|\Delta u|^2+|\Delta w|^2)\,dx
\nag\\&\hskip.4cm+
K_0\sum_{j=1}^J\int_{\Omega_j}|v\varphi_jm^j\cdot\nabla\bar f|+|z\varphi_jm^j\cdot\nabla\bar h|\,dx\\&\hskip.4cm+
K_0|\lambda|^{-1}(||Z||_\cH^2+||Z||_\cH||U||_\cH)+K_0\int_{\omega_1}(|v|^2+|z|^2)\,dx\nag\\&\hskip.4cm+K_0|\lambda|^{-1}||U||_{\cal H}^{3\over2}||Z||_{\cal H}^{1\over2}+
K_0||U||_{\cal H}^{1\over2}||Z||_{\cal H}^{1\over2}\left(\int_{\O}a(|D^2(u+w)|^2)\,dx\right)^{1\over2}.\nag
\end{align}
We would like to draw the reader's attention of the fact that
\begin{itemize}
    \item the terms involving the right hand sides in \eqref{xae2} are not fully estimated yet; estimating those terms is much simpler in the present context;
    \item we kept the localized terms on $\omega_1$; to estimate the localized kinetic and potential energies, we are going to use a suitable smooth cut-off function since the damping coefficient $a$ is no longer smooth.
\end{itemize}
Now, we are going to estimate the integral terms in the last inequality. We shall start with the last integral
\begin{align}\label{yae3}
\int_{\O}a(|D^2(u+w)|^2)\,dx\leq K_0(|\Delta u|_2^2+|\Delta w|_2^2)\leq K_0||Z||_\cH^2.
\end{align}
Applying Cauchy-Schwarz and Sobolev embeddings, we get
\begin{align}\label{yae4}
\sum_{j=1}^J\int_{\Omega_j} |g(\varphi_jm^j\cdot\nabla\bar u|+|\ell(\varphi_jm^j\cdot\nabla\bar w)|\,dx&\leq K_0(|g|_2|\nabla u|_2+|\ell|_2|\nabla w|_2)\nag\\&\leq K_0||U||_\cH||Z||_\cH,
\end{align}
and 
\begin{align}\label{yae5}
\sum_{j=1}^J\int_{\Omega_j}|v\varphi_jm^j\cdot\nabla\bar f|+|z\varphi_jm^j\cdot\nabla\bar h|\,dx
&\leq K_0(|v|_2|\nabla f|_2+|z|_2|\nabla h|_2)\nag\\&\leq K_0||Z||_\cH||U||_\cH.
\end{align}
Collecting those estimates, the inequality \eqref{yae2} becomes, (keeping in mind that $|\lambda|>1$):
\begin{align}\label{yae6}
||Z||_\cH^2&\leq K_0 ||Z||_\cH||U||_\cH+K_0\int_{\omega_1}(|\Delta u|^2+|\Delta w|^2)\,dx
\nag\\&\hskip.4cm+
K_0\int_{\omega_1}(|v|^2+|z|^2)\,dx\\&\hskip.4cm+K_0||U||_{\cal H}^{3\over2}||Z||_{\cal H}^{1\over2}+
K_0||U||_{\cal H}^{1\over2}||Z||_{\cal H}^{3\over2},\nag
\end{align}where the term involving $||Z||_\cH^2$ is absorbed to the left by choosing $|\lambda|$ large enough.\\
Using Young inequality, we readily derive from \eqref{yae6}:
\begin{align}\label{yae7}
||Z||_\cH^2\leq K_0 ||U||_\cH^2+K_0\int_{\omega_1}(|\Delta u|^2+|\Delta w|^2)\,dx
+
K_0\int_{\omega_1}(|v|^2+|z|^2)\,dx.
\end{align}
Now, we are going to introduce a cut-off function that will enable us to estimate the localized kinetic energy. Once this is done, we shall estimate the localized potential energy. To this end, let $\eta$ in $C^2(\bar\Omega)$ with
$$0\leq\eta\leq a_0\text{ in }\Omega,\quad \eta=a_0 \text{ in }\omega_1 \text{ and }\eta=0\text{ in }\Omega\setminus\omega,$$where $a_0$ is given by \eqref{coerc}. notice that as defined, the function $\eta$ satisfies: $\eta\leq a$ almost everywhere in $\Omega$.\\
We shall estimate the localized energy; the proof is divided into two steps. In the first step we shall estimate the localized kinetic energy, then in the second step, we shall estimate the localized potential energy and complete the proof.\\

\noindent
{\bf Step 1: Estimating the localized kinetic energy.} The main idea is the same as in the proof
in the last section. However, given that the function $a$ is no longer smooth, we cannot use it as we did in \eqref{ez81a}.\\
Using the cut-off function $\eta$, it follows
\begin{align}\label{yae8}
\int_{\omega_1}(|v|^2+|z|^2)\,dx\leq\frac{1}{a_0}\int_{\Omega}\eta^2(|v|^2+|z|^2)\,dx=\frac{1}{a_0}\int_{\Omega}\eta^2|v+z|^2\,dx-\frac{2}{a_0}\Re\int_{\Omega}\eta^2v\bar z\,dx.
\end{align}
Using interpolation as we did in Step 1 of the last section, we get
\begin{align}\label{yae9}
|\eta(v+z)|_2&\leq K_0||\eta(v+z)||_{-2}^{1\over3}||\eta(v+z)||_{1}^{2\over3}\nag\\&
\leq K_0(||v||_{-2}+||z||_{-2})^{1\over3}(||Z||_\cH+||U||_{\cal H}^{1\over2}||Z||_{\cal H}^{1\over2})^{2\over3},
\end{align}by following \eqref{ez9} and using the first inequality in \eqref{ey2}.\\
To estimate the $H^{-2}$-norm now, we note that a rough estimate of the term involving the damping is enough. Therefore using the second and last equations in \eqref{xae2}, we derive
\begin{align}\label{yae10}
|\lambda|||v||_{-2}&\leq K_0(|\Delta u|_2+||\text{div}(a\nabla(v+z)||_{-1}+|g|_2)\nag\\&\leq K_0(||Z||_\cH+||U||_{\cal H}^{1\over2}||Z||_{\cal H}^{1\over2}+||U||_\cH)
\end{align}
and 
\begin{align}\label{yae11}
|\lambda|||z||_{-2}&\leq K_0(|\Delta w|_2+||\text{div}(a\nabla(v+z)||_{-1}+|\ell|_2)\nag\\&\leq K_0(||Z||_\cH+||U||_{\cal H}^{1\over2}||Z||_{\cal H}^{1\over2}+||U||_\cH)
\end{align}
Gathering \eqref{yae9} to \eqref{yae11}, it follows
\begin{align}\label{yae12}
|\eta(v+z)|_2
&\leq K_0|\lambda|^{-{1\over3}}(||Z||_\cH+||U||_{\cal H}^{1\over2}||Z||_{\cal H}^{1\over2}+||U||_\cH)^{1\over3}(||Z||_\cH+||U||_{\cal H}^{1\over2}||Z||_{\cal H}^{1\over2})^{2\over3}\nag\\&\leq K_0(|\lambda|^{-{1\over3}}||Z||_\cH+||U||_{\cal H}^{1\over2}||Z||_{\cal H}^{1\over2}+||U||_{\cal H}^{1\over3}||Z||_{\cal H}^{2\over3}+||U||_{\cal H}^{2\over3}||Z||_{\cal H}^{1\over3}).
\end{align}
Combining \eqref{yae8} and \eqref{yae12}, we derive
\begin{align}\label{yae13}
\int_{\Omega}\eta^2(|v|^2+|z|^2)\,dx&\leq K_0(|\lambda|^{-{2\over3}}||Z||_\cH^2+||U||_{\cal H}||Z||_{\cal H}+||U||_{\cal H}^{2\over3}||Z||_{\cal H}^{4\over3}+||U||_{\cal H}^{4\over3}||Z||_{\cal H}^{2\over3})\nag\\&\hskip.4cm-\frac{2}{a_0}\Re\int_{\Omega}\eta^2v\bar z\,dx.
\end{align}
To complete this step, it remains to estimate the integral term in \eqref{yae13}.\\
Proceeding as in Step 1 of last section, replacing $a$ by $\eta^2$, we find
\begin{align}\label{yae14}
\Re\int_\Omega \eta^2v\bar z\,dx&=\frac{1}{c-d}\Bigg[\Re\frac{1}{i\lambda}\int_\Omega \eta^2\Big[cd(\Delta u\Delta\bar h-\Delta \bar w\Delta f) +cg\bar z+d\ell v\Big]\,dx\nag\\&\hskip.2in
-\Re\frac{cd}{i\lambda}\int_\Omega\Delta\bar w(2\nabla (\eta^2)\cdot \nabla v+v\Delta (\eta^2))+\Delta u(2\nabla (\eta^2)\cdot \nabla \bar z+\bar z\Delta (\eta^2))\,dx\\&\hskip.2in- \Re\frac{1}{i\lambda}\int_\Omega a\Big[c\nabla(v+z)\cdot(\eta^2\nabla\bar z+\bar z\nabla (\eta^2))+d\nabla(\bar v+\bar z)\cdot(\eta^2\nabla v+v\nabla (\eta^2))\Big]\,dx.\nag
\Bigg]\end{align}
Notice that the three terms in the right side of \eqref{yae14} can be estimated as in Step 1 of the last section, yielding

\begin{align}\label{yae15}
&\left|\Re\int_\Omega \eta^2v\bar z\,dx\right|\leq  K_0(|\lambda|^{-\frac{1}{2}}||Z||_{\cH}^2+||U||_{\cH}^{\frac{1}{2}}||Z||_{\cH}^\frac{3}{2}+||Z||_{\cH}||U||_{\cH}),\end{align}
where we have kept the power of $|\lambda|$ only where it is necessary to be kept.\\ 
Collecting \eqref{yae13} and \eqref{yae15} leads to
\begin{align}\label{yae16}
\int_{\Omega}\eta^2(|v|^2+|z|^2)\,dx&\leq K_0(||U||_{\cal H}||Z||_{\cal H}+||U||_{\cal H}^{2\over3}||Z||_{\cal H}^{4\over3}+||U||_{\cal H}^{4\over3}||Z||_{\cal H}^{2\over3})\nag\\&\hskip.4cm+K_0(|\lambda|^{-\frac{1}{2}}||Z||_{\cH}^2+||U||_{\cH}^{\frac{1}{2}}||Z||_{\cH}^\frac{3}{2}).
\end{align}
First, combining \eqref{yae7} and \eqref{yae16}, then using Young inequality and the definition of the cut-off function $\eta$ yield
\begin{align}\label{yae17}
||Z||_\cH^2\leq K_0 ||U||_\cH^2+K_0\int_{\omega_1}(|\Delta u|^2+|\Delta w|^2)\,dx\leq K_0 ||U||_\cH^2+K_0\int_{\Omega}\eta^2(|\Delta u|^2+|\Delta w|^2)\,dx.
\end{align}

\noindent
{\bf Step 2: Estimating the localized potential energy.} Multiply the second equation of \eqref{xae2} by $\eta^2\bar u$ and its last equation by $\eta^2\bar w$, then integrate over $\Omega$ and apply Green's formula to derive
\begin{align}\label{yae18}
\int_\Omega \eta^2(d|\Delta u|^2+c|\Delta w|^2)\,dx&=\int_\Omega \eta^2 (|v|^2+|z|^2)\,dx-d\Re\int_\Omega\Delta u(2\nabla (\eta^2)\cdot \nabla\bar u+\bar u\Delta(\eta^2))\,dx\nag\\&\hskip.2in -\Re\int_\Omega a\nabla(v+z)\cdot\nabla(\bar u\eta^2+\bar w\eta^2)\,dx+\Re\int_\Omega \eta^2 (g\bar u+\ell\bar w)\,dx\\&\hskip.4cm
-c\Re\int_\Omega\Delta w(2\nabla (\eta^2)\cdot \nabla\bar w+\bar w\Delta(\eta^2))\,dx\nag\\&\hskip.4cm+\Re\int_\Omega \eta^2 (v\bar f+z\bar h)\,dx\nag
\end{align}
The application of Cauchy-Schwarz and Poincar\'e inequalities as well as an elementary interpolation inequality lead to
\begin{align}\label{yae19}
&\left| d\Re\int_\Omega\Delta u(2\nabla (\eta^2)\cdot \nabla\bar u+\bar u\Delta(\eta^2))\,dx\right|+\left|c\Re\int_\Omega\Delta w(2\nabla (\eta^2)\cdot \nabla\bar w+\bar w\Delta(\eta^2))\,dx\right|\nag\\&\leq K_0(|\Delta u|_2|\nabla u|_2+|\Delta w|_2|\nabla w|_2)\leq K_0||Z||_\cH(|u|_2^{1\over2}|\Delta u|_2^{1\over2}+|w|_2^{1\over2}|\Delta w|_2^{1\over2})\\&\leq K_0(|\lambda|^{-{1\over2}}||Z||_\cH^2+||U||_{\cH}^{\frac{1}{2}}||Z||_{\cH}^\frac{3}{2}),\nag
\end{align}where we have used the first and third equation in \eqref{xae2} to obtain the last inequality.\\
Similarly, and thanks to \eqref{dislaw}, one derives
\begin{align}\label{yae20}
\left|\Re\int_\Omega a\nabla(v+z)\cdot\nabla(\bar u\eta^2+\bar w\eta^2)\,dx\right|&\leq K_0|\sqrt{a}\nabla(v+z)|_2(|\Delta u|_2+|\Delta w|_2)\nag\\&\leq K_0||U||_{\cH}^{\frac{1}{2}}||Z||_{\cH}^\frac{3}{2}.
\end{align}
Applying Cauchy-Schwarz and if needed Poincar\'e inequality, we find
\begin{align}\label{yae21}
\left|\Re\int_\Omega \eta^2 (g\bar u+\ell\bar w)\,dx\right|+\left|\Re\int_\Omega \eta^2 (v\bar f+z\bar h)\,dx\right|\leq K_0||U||_\cH||Z||_\cH.
\end{align}
The combination of \eqref{yae18} to \eqref{yae21} yields
\begin{align}\label{yae22}
&\int_\Omega \eta^2(d|\Delta u|^2+c|\Delta w|^2)\,dx\leq\int_\Omega \eta^2 (|v|^2+|z|^2)\,dx\nag\\&\hskip.4cm+K_0(|\lambda|^{-{1\over2}}||Z||_\cH^2+||U||_\cH||Z||_\cH+||U||_{\cH}^{\frac{1}{2}}||Z||_{\cH}^\frac{3}{2})\\&\leq
K_0(||U||_{\cal H}||Z||_{\cal H}+||U||_{\cal H}^{2\over3}||Z||_{\cal H}^{4\over3}+||U||_{\cal H}^{4\over3}||Z||_{\cal H}^{2\over3})+K_0(|\lambda|^{-{1\over2}}||Z||_\cH^2+||U||_{\cH}^{\frac{1}{2}}||Z||_{\cH}^\frac{3}{2}),\nag
\end{align}where we have used \eqref{yae16} to get the last inequality.\\
Collecting \eqref{yae17} and \eqref{yae22}, then applying Young inequality, it follows
\begin{align}\label{yae17bis}
||Z||_\cH^2\leq K_0 ||U||_\cH^2,
\end{align}which completes the proof of Theorem \ref{expstab}.\qed

\end{document}